\documentclass[11pt]{amsart}

\usepackage[top=1in, left=1in, right=1in, bottom=1in]{geometry}

\usepackage{amsmath,amssymb,amsthm,amsfonts}
\usepackage{paralist}
\usepackage{booktabs}
\usepackage[usenames,dvipsnames]{color}
\usepackage{comment}
\usepackage{graphicx,epsf}
\usepackage{graphics}
\usepackage{epsfig}
\usepackage{epstopdf}
\usepackage{psfrag}
\usepackage[shortlabels]{enumitem}

\usepackage{amsmath,amsfonts,amssymb,latexsym,amsthm}
\usepackage{graphicx,epsf}
\usepackage{amsthm}
\usepackage{multicol}
\usepackage{ytableau}

\usepackage{url}
\usepackage{hyperref} 

\usepackage{tikz}
\usetikzlibrary{automata,positioning,arrows}

\usepackage[colorinlistoftodos]{todonotes}

\RequirePackage{cleveref}
\usepackage{hypcap}
\hypersetup{colorlinks=true, citecolor=darkblue, linkcolor=darkblue}
\definecolor{darkblue}{rgb}{0.0,0,0.7}
\newcommand{\darkblue}{\color{darkblue}}

\definecolor{darkred}{rgb}{0.68,0,0}

\definecolor{darkgreen}{rgb}{0,.38,0}

\newcommand{\defn}[1]{\emph{\darkblue #1}}

\newlength{\ml}
\newtheorem{thm}{Theorem}[section]
\newtheorem{claim}[thm]{Claim}
\newtheorem{Def}[thm]{Definition}
\newtheorem{lemma}[thm]{Lemma}

\newtheorem{Cor}[thm]{Corollary}
\newtheorem{cor}[thm]{Corollary}
\newtheorem{prop}[thm]{Proposition}

\newtheorem{conj}[thm]{Conjecture}

\newtheorem{ex}[thm]{Example}

\theoremstyle{plain}

\theoremstyle{definition}

\newtheorem{rem}[thm]{Remark}

\numberwithin{equation}{section}

\newcommand{\la}{\lambda}

\newcommand{\nn}{\mathbb N}

\newcommand{\ga}{\gamma}

\newcommand{\al}{\alpha}
\newcommand{\be}{\beta}

\newcommand{\ssu}{\subset}

\def\.{\hskip.06cm}
\def\ts{\hskip.03cm}

\def\SSYT{ {\text {\rm SSYT}  } }

\def\P{{{\rm{\textsf{P}} }}}

\def\NP{{{\rm{\textsf{NP}} }}}
\def\coNP{{{\rm{\textsf{coNP}} }}}

\def\bba{{\text{\bf a}}}

\def\bbc{{\text{\bf c}}}

\def\bbv{{\text{\bf v}}}

\def\<{\langle}
\def\>{\rangle}

\def\det{\mathrm{det}}

\def\la{\lambda}

\begin{document}

\title{The Newton polytope of the Kronecker product }

\author{Greta Panova}

\address{Department of Mathematics, University of Southern California, 
Los Angeles, CA 90089}

\email{gpanova@usc.edu}
\urladdr{https://sites.google.com/usc.edu/gpanova/}

\author{Chenchen Zhao}
\email{zhao109@usc.edu}
\urladdr{https://sites.google.com/view/chenchen-zhao/}

\date{\today }


\begin{abstract}
We study the Kronecker product of two Schur functions $s_\lambda\ast s_\mu$, defined as the image of the characteristic map of the product of two $S_n$ irreducible characters. We prove special cases of a conjecture of Monical--Tokcan--Yong that its monomial expansion has a saturated Newton polytope. Our proofs employ the Horn inequalities for positivity of Littlewood-Richardson coefficients and imply necessary conditions for the positivity of Kronecker coefficients. 
\end{abstract}

\maketitle

\section{Introduction}

The Kronecker coefficients $g(\la,\mu,\nu)$ of the symmetric group present an 85 year old mystery  in algebraic combinatorics and representation theory.  They are  defined as the multiplicities of an irreducible $S_n$-module $\mathbb{S}_\nu$ in the tensor product of two other irreducibles: $\mathbb{S}_\la \otimes \mathbb{S}_\mu$.  Originally introduced by Murnaghan in 1938~\cite{Mur38,Mur56}, the question of their efficient computation has been reiterated many times since the 1980s, see for example works of~\cite{Las79,GR85,remm:89}. Stanley's 10th open problem in algebraic combinatorics~\cite{Sta00} is to find a manifestly positive combinatorial interpretation for the Kronecker coefficients. Yet, over the years, very little progress on this problem has been made and only for special cases. We refer to~\cite{Pan22} for an overview of these results, as well as other open problems related to the Kronecker coefficients which justify the ``mystery'' label. Their importance has been reinforced by their role in Geometric Complexity Theory, a program aimed at establishing computational lower bounds and ultimately separating complexity classes like $\P$ vs $\NP$, see~\cite{Pan23} and references therein. While no positive combinatorial formula exists, we also lack understanding for when such coefficients would be positive. Recent work, see~\S\ref{ss:complexity}, has cast doubt on the possibility of answering these questions in a ``nice'' way. 

In a different direction, \cite{Yong} initiated the study of the Newton polytopes of important polynomials in algebraic combinatorics. It has been established that some of the main polynomials of interest, such as the Schur and Schubert polynomials, have \emph{saturated Newton polytope} (SNP). See \S\ref{ss:snp_lit} for more details and background.

\begin{Def}\label{def:snp} A multivariate polynomial with nonnegative coefficients $f(x_1,\ldots,x_k)=\sum_{\al} c_\al x^\al$ has a saturated Newton polytope (SNP) if the set of points $M_k(f):=\{(\al_1,\cdots,\al_k): c_\al>0\}$ coincides with its convex hull in $\mathbb{Z}^k$. 
\end{Def}

Given a symmetric function $f$, let $f(x_1,\dots, x_k)$ denote the specialization of $f$ to the variables $x_1,\dots,x_k$ that sets $x_m = 0 $ for all $ m \ge k+1.$
\begin{Def}\label{def:snp_sym_fun} A symmetric function $f$ has a saturated Newton polytope (SNP) if $f(x_1,\dots, x_k)$ has a SNP for all $k \ge 1.$\end{Def}

\subsection{SNP for the Kronecker product} 
The Kronecker coefficients of $S_n$, denoted by $g(\la,\mu,\nu)$, give the multiplicities of one Specht module in the tensor product of the other two, namely
$$\mathbb{S}_\la \otimes \mathbb{S}_\mu = \oplus_{\nu \vdash n} \mathbb{S}_\nu^{\oplus g(\la,\mu,\nu)}.$$ The Kronecker (tensor) product $\ast$ of symmetric functions is defined on the Schur basis as 
$$s_\la \ast s_\mu := \sum_\nu g(\la,\mu,\nu) s_\nu,$$
and extended by linearity. It is equivalent to the inner product of $S_n$ characters under the characteristic map. Via Schur-Weyl duality the Kronecker coefficients can be interpreted in the context of the representation theory of the general linear group, namely $g(\la,\mu,\nu)$ is the dimension of the projection of $V_\la( A\otimes B)$ to $V_\mu(A) \times V_\nu(B)$, where $V_{\la},V_{\mu},V_{\nu}$ are irreducible Weyl modules and $A$, $B$ are vector spaces. This viewpoint is important in the applications to Quantum Information Theory, see e.g.~\cite{christandl2006spectra, CHM}, and motivates the consideration of partitions of fixed length 2 or 3.  

\begin{conj}[\cite{Yong}]\label{conj:snp}
The Kronecker product $s_\la*s_\mu = \sum_\nu g(\la,\mu,\nu)s_\nu$ has a saturated Newton polytope. 
\end{conj}

Here we prove this conjecture for partitions of lengths 2 and 3 and various truncations.
\begin{thm}\label{thm:2row}
Let $\la,\mu \vdash n$ with $\ell(\la)\leq 2, \ell(\mu)\leq 3$, and $\mu_1 \geq \la_1$ then
$s_\la*s_\mu(x_1,\ldots,x_k)$ has a saturated Newton polytope for every $k\in \mathbb{N}$.
\end{thm}

This theorem follows from the fact that the Kronecker product in these cases contains a term $s_\nu$, with $\nu$ dominating all other partitions appearing in the product. Consequently, the degree vectors of the monomials which appear in that Kronecker product are the integer points $(a_1, \ldots, a_k)$ for which when sorted in weakly decreasing order $sort(a_1, \ldots, a_k) \preceq \nu$ in the dominance order. This ensures that the polytope is saturated. However, it is not always the case that there is a unique maximal term with respect to the dominance order. The first instance where no such dominant partition exists is covered in the following theorem.

\begin{thm}\label{thm:main}
Let $\la,\mu \vdash n$ with $\ell(\la)\leq 3$ and $\ell(\mu)\leq 2$. Then 
$s_\la * s_\mu(x_1,x_2,x_3)$ has a saturated Newton polytope. 
\end{thm}

The difficulty with this problem in general lies in the lack of explicit criteria for the positivity of the Kronecker coefficients. Instead of working directly with the Kronecker coefficients, we express the Kronecker product in the monomial basis in terms of sums of products of multi-Littlewood-Richardson coefficients. We then use the Horn inequalities, which determine when a Littlewood-Richardson coefficient would be nonzero, to construct a polytope $\mathcal{P}(\la,\mu; \mathbf{a})$ parametrized by partitions $\la,\mu$ and a composition $\mathbf{a}=(a_1,\ldots,a_k)$. 
A monomial $x_1^{a_1}x_2^{a_2}\cdots x_k^{a_k}$ appears in $s_\la \ast s_\mu(x_1,\ldots,x_k)$ if and only if $\mathcal{P}(\la,\mu;\mathbf{a})$ has an integer point. Using this construction we can infer the following.

\begin{prop}\label{cor:int_point_snp}
   Let $\mu,\la \vdash n$. The Kronecker product $s_\la \ast s_\mu (x_1,\ldots,x_k)$ has a saturated Newton polytope if and only if for every $\bba\in \mathbb{Z}^k$ the polytope $\mathcal{P}(\la,\mu;\bba)$ is either empty or has an integer point. 
\end{prop}
Since $g(\la,\la,(n))=1$, we have that $s_\la* s_\la=s_{(n)}+\cdots$ and contains every monomial of total degree $n$. This implies that $\mathcal{P}(\la,\la;\mathbf{a})$ is always nonempty and has an integer point. However, it is far from clear how to explicitly characterize when $\mathcal{P}(\la,\mu;\bba) \neq \emptyset$ once $\mu \neq \la$ and the number of variables $k$ grows. Additionally, determining whether there is an integer point, in this case, is likely hard from computational complexity point of view, see~\S\ref{ss:complexity}. It is also not apparent whether these polytopes have any integer vertices, as the associated inequalities often result in many non-integral ones.

While Conjecture~\ref{conj:snp} may not necessarily hold in full generality, its limiting version does:

\begin{thm}\label{thm:limit}
Let $\la,\mu$ be partitions of the same size and $k \in \mathbb{N}$. Then the set of points 
$$\bigcup_{p=1}^\infty \frac{1}{p} M_k(s_{p\la}\ast s_{p\mu})$$ is a convex subset of $\mathbb{Q}^k$. 
\end{thm}

This is not surprising given that the set of triples $\frac{1}{|\la|}(\la,\mu,\nu)$ for which there is a $p$, such that $g(p\la,p\mu,p\nu)>0$, forms a polytope known as the Moment polytope, see~\cite{vergne2017inequalities,burgisser2017membership}.  

\subsection{Positivity implications}

Suppose that  $g(\la,\mu,\al) >0$ and $g(\la,\mu,\be)>0$ for some partitions $\al,\be$. Then the monomials with powers $\al$ and $\be$ appear in $s_{\la}*s_\mu$. Suppose that $\gamma = t\al+(1-t)\be \in \mathbb{Z}^k$ for some $t=\frac{p}{q} \in \mathbb{Q}$ with $p<q$. The SNP property would imply that $\gamma$ appears as a monomial, and thus there exists a partition $\theta \succ \gamma$ such that $g(\la,\mu,\theta)>0$. 
By the semigroup property, if $g(p\lambda, p\mu, p\alpha) > 0$ and $g((q-p)\lambda, (q-p)\mu, (q-p)\beta) > 0$, then it follows that $g(q\lambda, q\mu, q\gamma) > 0$. However, the Kronecker coefficients do not, in general, possess the saturation property (see \S\ref{ss:kron_sat}), so we cannot expect $g(\la,\mu,\ga)>0$. We can generalize the above reasoning into the following.

\begin{prop}\label{cor:snp_positivity}
Suppose that $s_\la * s_\mu$ has a saturated Newton polytope. Then for every collection of partitions $\al^1,\al^2,\ldots$ such that $g(\la,\mu,\al^i)>0$ and $\sum_i t_i \al^i $ has integer parts for some $t_i \in  [0,1]$ with $t_1+t_2+\cdots=1$, there exists a partition $\theta \succeq sort(\sum_i t_i \al^i)$ in the dominance order, such that $g(\la,\mu,\theta)>0$. 
\end{prop}

Our methods and the Horn inequalities also give some necessary conditions for a Kronecker coefficient to be positive. We cannot expect easy necessary and sufficient criteria for positivity since this decision problem is $\NP$-hard by~\cite{IMW}. The general statement is Theorem~\ref{prop:kron_pos_gen} stated in Section~\ref{s:pos} in terms of the so-called LR-consistent triples. We illustrate the criteria with a simplified version below in the case of one two-row partition. 

\begin{prop}[Corollary~\ref{cor:kron_pos_2}]
    Suppose that $g(\la,\mu,\nu)>0$ and $\ell(\mu)=2$, $k=\ell(\la)$. Then there exist nonnegative integers $y_i \in [0, \lfloor \la_i/2 \rfloor ]$ for $i\in[k]$, such that
    \begin{align}
    \sum_{i\in A \cup C} \la_i +\sum_{i \in B }  y_i -\sum_{i\in C} y_i  \leq   \min\{\sum_{j \in J} \mu_j, \sum_{j \in J} \nu_j\}
\end{align}
for all triples of mutually disjoint sets $A \sqcup B \sqcup C \subset[k]$ and $J=\{1,\ldots,r,r+2,\ldots,r+b+1\}$ or $J= \{1,\ldots,r+b-1,r+2b\}$, where $r=2|A|+|C|$ and $b=|B|$.
\end{prop}

\subsection{}\label{ss:snp_lit}
The topic of studying Newton polytopes in algebraic combinatorics has quickly gained momentum. Already in~\cite{Yong} it was shown that, besides Schur and skew Schur functions, the Stanley symmetric functions and Macdonald polynomials have SNP. Since~\cite{Yong} there has been a flurry of activity, studying the SNP property for Schubert polynomials, see~\cite{fink2018schubert,castillo2023double}, Grothendieck polynomials, see~\cite{escobar2017newton, meszaros2020generalized} and chromatic symmetric functions, see~\cite{monicalthesis,matherne2022newton}. In the case of Schubert and double Schubert polynomials, the conjectures were completely resolved, revealing a nice geometric structure. For Grothendieck polynomials, the conjectures were resolved for Grassmannian permutations but remain open in general. In the case of chromatic symmetric functions, the conjectures hold for claw-free incomparability graphs and are false in other cases. Naturally, saturated Newton polytopes are a necessary condition for stronger nice properties like the (normalized) polynomials being Lorenzian, or having log-concave coefficients. It is also useful in the study of positivity questions and Schur expansions. 

\subsection{Paper outline}
We define the necessary background objects and tools in Section~\ref{s:def}. In Section~\ref{s:2-3-row}, we study cases where there is a unique maximal dominant partition in the Kronecker product. In this section, we prove Theorem~\ref{thm:2row} and one case of Theorem~\ref{thm:main}. In Section~\ref{s:LRs}, we reformulate the problem in terms of multi-Littlewood-Richardson coefficients, recall the Horn inequalities and define the polytope $\mathcal{P}(\la,\mu;\bba)$. In Section~\ref{s:int_points}, we consider partitions of length 2 and 3 and show that when $\mathcal{P}(\la,\mu;\bba)$ is nonempty it has an integer point proving Theorem~\ref{thm:main}. We study positivity properties of Kronecker coefficients in Section~\ref{s:pos}. In Section~\ref{s:final}, we discuss related aspects of the Kronecker SNP problem, including some observations for the analogous problem on plethysms in~\S\ref{ss:plethysm}.

\section{Definitions and tools}\label{s:def}

\subsection{Basic notions from algebraic combinatorics}
We use standard notation from~\cite{Mac} and~\cite[$\S$7]{S1}
throughout the paper.

Let $\la=(\la_1,\la_2,\ldots,\la_\ell)$ be a \defn{partition}
of size $n:=|\la|=\la_1+\la_2+\ldots+\la_\ell$, where
$\la_1\ge \la_2 \ge \ldots \ge\la_\ell\ge 1$.  We write
$\la\vdash n$ for such a partition. The length of~$\la$
is denoted $\ell(\la):=\ell$.  Let $\la+\mu$ denote a partition
$(\la_1+\mu_1,\la_2+\mu_2,\ldots)$

\ytableausetup{boxsize=2ex}

A \defn{Young diagram} of
\emph{shape}~$\la$ is an arrangement of squares
$(i,j)\ssu \nn^2$ with $1\le i\le \ell(\la)$
and $1\le j\le \la_i$.  Let $\la\vdash n$.
A \defn{semistandard Young tableau} (SSYT) $T$ of
\emph{shape}~$\la$ and \emph{type}~$\al$ is an
arrangement of $\al_k$ many integers~$k$ in squares
of~$\la$, which weakly increase along rows and strictly
increase down columns, i.e. $T(i,j) \leq T(i,j+1)$ and $T(i,j) \leq T(i+1,j)$. For example, 
\[\ytableaushort{11244,2235,45}\] is an SSYT of shape $\la=(5,4,2)$ and type (weight) $\al=(2,3,1,3,2)$.
Denote by $\SSYT(\la,\al)$ the
set of such tableaux, and \ts
$K_{\la,\al} = \bigl|\SSYT(\la,\al)\bigr|$ \ts the
\defn{Kostka number}. We have that $K_{\la,\al}=K_{\la,sort(\al)}$, where $sort(\al)$ is the partition obtained by listing the parts of $\al$ in weakly decreasing order.  When $\la$ and $\al$ are both partitions, we have $K_{\la,\al}>0$ if and only if $\la \succ \al$ in the \defn{dominance order} on partitions, which is defined as follows:
$$\la \succ \al \Longleftrightarrow \la_1+\cdots+\la_i \geq \al_1+\cdots +\al_i \text{ for all } i.$$
 A \defn{standard Young tableau} (SYT) of shape $\la$ is an SSYT of type $(1^n)$.

The irreducible representations of the \emph{symmetric group} $S_n$ are the \defn{Specht modules} $\mathbb{S}_\la$ and are indexed by partitions $\la \vdash n$. 
We have that $\mathbb{S}_{(n)}$ is the trivial representation assigning to every $w\in S_n$ the value $1$ and $\mathbb{S}_{1^n}$ is the sign representation. 

The \defn{Kronecker coefficients} of $S_n$, denoted by $g(\la,\mu,\nu)$ give the multiplicities of one Specht module in the tensor product of two other irreducibles, namely
$$\mathbb{S}_\la \otimes \mathbb{S}_\mu = \oplus_{\nu \vdash n} \mathbb{S}_\nu^{\oplus g(\la,\mu,\nu)}.$$ 
  
The irreducible polynomial representations of  $GL_N(\mathbb{C})$ are the \defn{Weyl modules} $V_\la$ and are indexed by all partitions with $\ell(\la) \leq N$. Their characters are  the Schur polynomials $s_\la(x_1,\ldots,x_N)$, where $x_1,\ldots,x_N$ are the eigenvalues of $g \in GL_N(\mathbb{C})$.

We will use the standard bases for the ring of symmetric functions $\Lambda$: the \defn{monomial} symmetric functions
$$m_\al(x_1,x_2,\ldots,x_k) = \sum_{\sigma} x_{\sigma_1}^{\al_1}x_{\sigma_2}^{\al_2}\cdots ,$$ where the sum goes over all  permutations $\sigma$ giving different monomials. 

The \defn{Schur functions} $s_\la$ can be defined as the generating functions for SSYTs of shape $\la$, i.e.
\begin{equation}\label{eq:schur_ssyt}
s_\la =  \sum_\al K_{\la \al}m_\al.
\end{equation}
We will also use the \defn{homogeneous} symmetric functions $h_\la$ defined as 
$h_k := s_{(k)}=\sum_{i_1 \leq \cdots \leq i_k} x_{i_1}\cdots x_{i_k}$ and $h_\la:= h_{\la_1} h_{\la_2}\cdots$. They are dual to the monomial symmetric functions under the inner product of the ring of symmetric functions $\Lambda$:
$$\langle h_\la , m_\mu \rangle = \delta_{\la,\mu} = \langle s_\la, s_\mu \rangle,$$
where $\delta_{\la,\mu}=1$ if and only if $\la=\mu$ and is 0 otherwise. 

The \defn{Littlewood-Richardson} coefficients $c^{\la}_{\mu\nu}$ are defined as structure constants in $\Lambda$ for the Schur basis, and also as the multiplicities in the $GL$-module decomposition $V_\mu \otimes V_\nu = \oplus_\la V_\la^{c^{\la}_{\mu\nu}}$. We have
$$s_\mu s_\nu = \sum_\la c^{\la}_{\mu\nu} s_\la.$$

They can be evaluated by the \defn{Littlewood-Richardson rule} as a positive sum of skew SSYT of shape $\la/\mu$ and type $\nu$ whose reverse reading word is a ballot sequence. The reverse reading word of a tableau is the sequence obtained by reading the tableau from top to bottom and right to left. A ballot sequence is a sequence of positive integers such that in every prefix of the sequence and for every $i$, there are at least as many occurrences of $i$ as there are of $i + 1$. For example, we have $c^{(6,4,3)}_{(3,1),(4,3,2)}=2$ as there are two LR tableaux, i.e., SSYT of shape $(6,4,3)/(3,1)$ and type $(4,3,2)$ whose reading words are ballot sequences. These tableaux are:
\begin{align*}
\ytableaushort{\none\none\none 111,\none 122,233}\, \text{ with reading word } 111221332
\text{, and } \ytableaushort{\none\none\none 111,\none 222,133}\, \text{ with reading word } 111222331. 
\end{align*}

Their positivity can be decided in poylnomial time as $c^{\la}_{\mu\nu} >0$ if and only if its corresponding polytope is nonempty, see~\cite{KT,MNS}. 
The \defn{multi-Littlewood-Richardson (multi-LR) coefficients} can be defined recursively as
$$c^\la_{\nu^1,\nu^2,\ldots,\nu^k} := \langle s_\la, s_{\nu^1}s_{\nu^2}\cdots s_{\nu^k}\rangle = \sum_{\tau^1,\tau^2,\ldots,\tau^k} c^{\la}_{\nu^1 \tau^1} c^{\tau^1}_{\nu^2 \tau ^2} \cdots c^{\tau^{k-1}}_{\nu^k \tau ^k},$$
where the sum is over all sequences of partitions $(\tau^1,\ldots,\tau^k)$, such that $|\tau^i| = |\tau^{i-1}|-|\nu^i|$ for $i=2,\ldots,k$ and $|\tau^1|=|\la|-|\nu^1|.$

\subsection{The Kronecker product}
The Kronecker product, denoted by $*$, of symmetric functions can be defined on the basis of the Schur functions and extended by linearity:
$$s_\la * s_\mu = \sum_\nu g(\la,\mu,\nu) s_\nu.$$

It is also ${\rm ch}(\chi^\la \chi^\mu)$, where $\chi^\la$ and  $\chi^\mu$ are the $S_n$ characters indexed by $\la$ and $\mu$, respectively, and ${\rm ch}$ is the Frobenius characteristic map. The Kronecker coefficients can be equivalently defined as the coefficients in the following expansion\begin{equation}\label{eq:kron_pleth}
s_\la[x\cdot y] = \sum_{\mu,\nu} g(\la,\mu,\nu) s_\mu(x) s_\nu(y)
\end{equation} where $[x\cdot y]:= (x_1y_1,x_1y_2,\ldots,x_2y_1,\ldots)$ denote all pairwise products of the two sets of variables. From this identity, we can derive the triple Cauchy identity, see e.g. \cite[Exercise 7.78]{stanley1997enumerative}:
\begin{equation}\label{eq:tripple_cauchy}
\sum_{\la,\mu,\nu} g(\la,\mu,\nu) s_\la(x) s_\mu(y) s_\nu(z) = \prod_{i,j,k} \frac{1}{1-x_iy_jz_k}.
\end{equation}

It is clear from the last equation that the Kronecker coefficient is invariant under the permutation of the three partitions, i.e. $g(\la,\mu,\nu) = g(\nu,\mu,\la)=\cdots$.

Via Schur-Weyl duality the Kronecker coefficients can be interpreted as the dimensions of $GL$ highest weight spaces, which then makes the following \defn{semigroup property}, apparent:
\begin{lemma}[\cite{CHM}]
    If $\al^1, \be^1, \ga^1 \vdash n$ and $\al^2,\be^2,\ga^2 \vdash m$ satisfy $g(\al^i,\be^i,\ga^i)>0$ for $i=1,2$, then $g(\al^1+\al^2, \be^1+\be^2, \ga^1+\ga^2) \geq \max\{ g(\al^1,\be^1,\ga^1), g(\al^2,\be^2,\ga^2)\}$.
\end{lemma}

Here we will be concerned with the monomial expansion. Since the homogeneous and monomial bases are orthogonal to each other, i.e. $\langle h_\la ,m_\mu \rangle = \delta_{\la,\mu}$, we have that 

\begin{align}\label{eq:kron_mon}
s_\la \ast s_\mu = \sum_\nu g(\la,\mu,\nu) s_\nu = \sum_{\nu,\al} g(\la,\mu,\nu) K_{\nu \al} m_\al = \sum_{\al \vdash n} \langle s_\la \ast s_\mu, h_\al\rangle m_\al.
\end{align}

In Section~\ref{s:LRs} we will see further ways of finding the monomial expansion.

\subsection{Newton polytopes}

Let $f(x_1,\ldots,x_k) = \sum_\al c_\al x^\al$ be a polynomial with nonnegative coefficients, i.e., $c_\al \ge 0$ for all $\al$, where $x^\al := x_1^{\al_1}\cdots x_k^{\al_k}$ and $\al \in \mathbb{Z}_{\geq 0}^k$ is the degree vector. We denote by 
$M_k(f) := \{ \al \in \mathbb{Z}_{\geq 0}^k \, : \; c_\al >0\}$ the set of vectors, for which the corresponding monomial appears in $f(x_1,\ldots,x_k)$. For brevity, we will say ``the monomial $\al$ appears in $f$''. We denote by $N_k(f):= Conv(M_k(f))$ the convex hull of $M_k(f)$, this is the \defn{Newton polytope} of $f(x_1,\ldots,x_k)$. We can restate Definition~\ref{def:snp} as follows.

\begin{Def}[\cite{Yong}]
A polynomial $f$ has a \defn{saturated Newton polytope} (SNP) if $M_k(f) = N_k(f)$.
\end{Def}

In particular, a polynomial $f$ has an SNP if and only if the following condition holds: 
\begin{align}\label{eq:snp}\tag{SNP}
 \text{For every $k+1$-tuple of compositions $(\al^1,\ldots,\al^{k+1})$, such that $c_{\alpha^i}>0$, and } \\ \notag
 \text{ weights $t_i \in[0,1]$, such that $t_1+\cdots+t_{k+1}=1$ and $\gamma := \sum_{i=1}^{k+1} t_i \alpha^i \in \mathbb{Z}^k$, we have $c_\gamma >0$.}
\end{align} By Caratheodory's theorem, any point in the convex hull of $M_k(f)$ can be written as a convex combination of at most $k+1$ points in $k$-dimensional space. Therefore, to show $M_k(f)$ is convex, it suffices to verify that every convex combination of $k+1$ points in $M_k(f)$ remains within $M_k(f)$.

As noted in~\cite{Yong} many of the important symmetric polynomials have SNP. Since Kostka coefficients $K_{\la\mu}$ are positive if and only if $\la \succeq \mu$ in the dominance order, we get an immediate characterization of $M_k(s_\la)$ and the following important statement. 

\begin{prop}[\cite{Yong}]\label{prop:schur_polytope}
The Schur polynomial $s_\la(x_1,\ldots,x_k)$ has a saturated Newton polytope and 
$$M_k(s_\la) = conv\{ (\la_{\sigma_1},\ldots,\la_{\sigma_k} ) \text{ for all } \sigma\in S_k\}.$$
\end{prop}

\section{Two and three-row partitions}\label{s:2-3-row}
Here we deduce the SNP property for certain cases from existing formulas. In the cases treated here we will see that 
there will be a unique partition $\la$, such that $g(\mu,\nu,\la)>0$ and if $g(\mu,\nu,\al)>0$ then $\la \succ \al$ and so $s_\mu \ast s_\nu$ will contain all monomials $\al \prec \la$, as observed in~\cite{Yong}.

\begin{lemma}\label{lem:unique_dom}
Suppose that $\mu\vdash n$ and $\nu \vdash n$ are such that there is a partition $\la$ with $g(\mu,\nu,\la) >0$ and for every $\tau$, such that $g(\mu,\nu,\tau)>0$ we have $\tau \prec \la$ in the dominance order. Then $M_k(s_\mu \ast s_\nu) = M_k(s_\la)=N_k(s_\la)$ for all $k$ and $s_\mu*s_\nu$ has a saturated Newton polytope.
\end{lemma}
\begin{proof}
Let $\al$ be such that $m_\al$ appears in $s_\mu \ast s_\nu$. By equation~\eqref{eq:kron_mon} we have that there must be a $\tau$, such that $g(\mu,\nu,\tau) K_{\tau, \al}>0$, so $K_{\tau, \al}>0$ and $\al \prec \tau$. Also, $g(\mu,\nu,\tau)>0$, thus by the hypothesis of this lemma, $\tau \prec \la$, so we get $\al \prec \tau \prec \la$. Thus $K_{\la,\al}>0$ and $m_\al$ appears in $s_\la$. The other direction follows since $s_\la$ is itself a term in the Kronecker product. 
\end{proof}

First, let $\ell(\nu)=\ell(\mu) =2$ and let the number of variables be arbitrary. Here we will show that 
$N_k(s_\nu*s_\mu) = N_k(s_\la)$ for a certain partition $\la$ for all $k$.

In~\cite{rosas}, Rosas computed the Kronecker product of Schur functions corresponding to two-row partitions. In particular, she showed that
\begin{prop}[{\cite[Corollary 5]{rosas}}]\label{cor:kron-coeff-two-row} Let $\al = (\al_1, \al_2), \be = (\be_1,\be_2)$, and $\ga = (\ga_1,\ga_2)$ be partitions of $n$. Assume that $\ga_2 \le \be_2\le \al_2$. Then $$g(\al,\be,\ga) = \max\{ y-x,0\},$$ where $x = \max\left(0,\left\lceil \frac{\be_2+\ga_2+\al_2-n}{2}\right\rceil\right)$ and $y = \left\lceil\frac{\be_2+\ga_2-\al_2+1}{2} \right\rceil$.
\end{prop}

\begin{lemma}\label{lem:maximal-two-row} Let $\lambda = (\lambda_1, \lambda_2)$, $\mu = (\mu_1,\mu_2)$, and $\nu = (\nu_1,\nu_2)$ be two-row partitions of $n$. Without loss of generality, suppose that $\mu_2 \ge \nu_2$. Then $\langle s_{\mu}\ast s_{\nu}, h_{\lambda}\rangle > 0$ if and only if $\lambda_2 \ge \mu_2-\nu_2 .$
\end{lemma}
By equation~\eqref{eq:kron_mon} this means that $m_\la$ appears with a nonzero coefficient in that Kronecker product. 
\begin{proof}

First, suppose that $\la_2 \geq \mu_2 - \nu_2$. We consider two cases.

 Case 1: $\lambda_2 \le \mu_2$. Then by Proposition~\ref{cor:kron-coeff-two-row}, applied with $\ga =\la, \be = \nu, \al=\mu$ we have that $y=\lceil \frac{\la_2 + \nu_2-\mu_2+1}{2} \rceil$, so  $g(\mu, \nu, \lambda) = 0$ whenever $\nu_2+\lambda_2-\mu_2 <0$. Also $g(\mu, \nu, \lambda) = 1$ when $\nu_2+\lambda_2-\mu_2 =0$ since $y=1$ and $\la_2+\mu_2+\nu_2 -n  = \la_2+\nu_2 -\mu_1 \leq \la_2+\nu_2 -\mu_2 =0$ so $x=0$. It follows that $\langle s_{\mu}\ast s_{\nu}, h_{\lambda}\rangle \geq \langle s_\mu \ast s_\nu, s_\la \rangle >0$ if $ \mu_2-\nu_2 \le \lambda_2 \le \mu_2.$

Case 2: $\lambda_2 > \mu_2$. Then $\mu_1 \geq \la_1$ and  $\mu \succ \lambda$. Applying Proposition~\ref{cor:kron-coeff-two-row} with $\ga=\nu, \be=\al=\mu$ we have $y =\lceil \frac{ \nu_2+1}{2} \rceil \geq  \lceil \frac{ \nu_2+\mu_2 - \mu_1}{2} \rceil=x$ with equality if and only if $\mu_2=\mu_1$ and $\nu_2$ is odd. However this case is impossible, since $n/2 \geq \la_2 >\mu_2$, we have $\mu_2<\mu_1$, so $g(\mu, \nu, \mu) >0$. As above, $\langle s_{\mu}\ast s_{\nu}, s_{\mu}\rangle > 0$ and since $K_{\mu,\la}>0$, then the monomial $m_\la$ appears.

In the opposite direction, consider $\lambda_2 < \mu_2-\nu_2.$ Then, $\lambda_2 < \mu_2$, and $\lambda_2+\nu_2-\mu_2 <0$, which implies that $g(\lambda,\mu,\nu) = 0$ by Corollary 2. Thus, $\langle s_{\mu}\ast s_{\nu}, s_{\lambda}\rangle = 0$ for all $\la$ with $\lambda_2 < \mu_2-\nu_2$. If $\langle s_\nu \ast s_\mu, h_\la \rangle >0$ for some $\la$ with $\la_2 < \mu_2-\nu_2$ then there must exist some partition $\theta \succ \la$, such that $ \langle s_\nu \ast s_\mu, s_\theta \rangle >0$. But as $\theta \succ \la$ we would have $\theta_2 \leq \la_2 <\mu_2 -\nu_2$, and thus reach a contradiction. 
\end{proof}

We now move to a more general case and invoke Theorem 4 from~\cite{rosas}. Adopting the notation from~\cite{rosas}, we treat the coordinate axes as if working with matrices, where the first entry corresponds to $(0, 0)$. Specifically, the point $(i, j)$ refers to  $i$-th row and $j$-th column.

Let $k$ and $l$ be positive integers representing the length and height of a rectangle, respectively. Define the function $\sigma_{k, l}:\mathbb{N} \to \mathbb{N}$, where $\sigma_{k, l}(h)$ counts the number of points in $\mathbb{N}^2$ inside a rectangle of dimensions $k\times l$ that can be reached from $(0,h)$ by moving any number of steps southwest or northwest. As shown in~\cite[Lemma 1]{rosas}, the function satisfies the following recursive definition:
\begin{equation}\label{sigma_eqn}\sigma_{k, l}(h)= \begin{cases}0, &  h<0 \\ \left\lfloor\left(\frac{h}{2}+1\right)^2\right\rfloor, & 0 \leq h<\min (k, l) \\ \sigma_{k, l}(s)+\left(\frac{h-s}{2}\right) \min (k, l), & \min (k, l) \leq h<\max (k, l) \\ \left\lceil\frac{k l}{2}\right\rceil-\sigma_{k, l}(k+l-h-4), &  h \text { is even and } \max (k, l) \leq h \\ \left\lfloor\frac{k l}{2}\right\rfloor-\sigma_{k, l}(k+l-h-4), & h \text { is odd and } \max (k, l) \leq h\end{cases},\end{equation}
where \[s = \begin{cases}
\min(k, l) - 2 & \text{if } h - \min(k, l) \text{ is even}, \\
\min(k, l) - 1 & \text{otherwise}.
\end{cases}\]

Let $R$ be the rectangle with vertices $(a, c),
(a+b,c), (a,c+d),$ and $(a+b,c+d)$. The function $\phi(a, b, c, d)(x, y)$ counts the number of points in $\mathbb{N}^2$ inside $R$ that can be reached from $(x,y)$ by moving any number of steps southwest or northwest. Rosas provided a concise form for this function as follows: 
\begin{equation}\label{phi_eqn}
    \phi(a, b, c, d)(x, y)= \begin{cases}\sigma_{b+1, d+1}(x+y-a-c), & 0 \leq y \leq c \\ \sigma_{b+1, y-c+1}(x-a)+\sigma_{b+1, c+d-y+1}(x-a)-\delta, & c<y<c+d \\ \sigma_{b+1, d+1}(x-y+c+d-a), & c+d \leq y\end{cases}
\end{equation}

where the auxiliary variable $\delta$ is defined as follows:
\begin{itemize}
    \item If $x<a$, then $\delta=0$;
    \item If $a \leq x \leq a+b$, then $\delta=\left\lceil\frac{x-a+1}{2}\right\rceil$;
    \item If $x>a+b$ then we consider two cases: If $x-a-b$ is even then $\delta=\left\lceil\frac{b+1}{2}\right\rceil$; otherwise, $\delta=\left\lfloor\frac{b+1}{2}\right\rfloor$.
\end{itemize}
Then there is the following formula for Kronecker coefficients of 2 two-row partitions.

\begin{thm}[{\cite[Theorem 4]{rosas}}]\label{thm:kron-coeff}
Let $\be, \ga$, and $\al$ be partitions of $n$, where $\be=\left(\be_{1}, \be_{2}\right)$ and $\ga=\left(\ga_{1}, \ga_{2}\right)$ are two two-row partitions and let $\al=\left(\al_{1}, \al_{2}, \al_{3}, \al_{4}\right)$ be a partition of length less than or equal to 4. Assume that $\ga_{2} \leq \be_{2}$. Then
$$
g(\be,\ga ,\al)=(\phi(a, b, a+b+1, c)-\phi(a, b, a+b+c+d+2, c))\left(\ga_{2}, \be_{2}+1\right),
$$
where $a=\al_{3}+\al_{4}, b=\al_{2}-\al_{3}, c=\min \left(\al_{1}-\al_{2}, \al_{3}-\al_{4}\right) \text { and } d=\left|\al_{1}+\al_{4}-\al_{2}-\al_{3}\right|$.
\end{thm}

\begin{ex} Consider $\be = (7,6)$, $\ga = (8,5)$ and $\al = (5, 4, 3, 1)$. 

By Theorem \ref{thm:kron-coeff}, $a=\al_{3}+\al_{4} = 4, b=\al_{2}-\al_{3}=1, c=\min \left(\al_{1}-\al_{2}, \al_{3}-\al_{4}\right)=1$, $ d=\left|\al_{1}+\al_{4}-\al_{2}-\al_{3}\right|=1$, and  hence, $g(\be,\ga,\al) = (\phi(4, 1, 6, 1)-\phi(4, 1, 9, 1))\left(5, 7\right).$

By equation \eqref{phi_eqn}, we have \[\phi(4, 1, 6, 1)(5,7) = \sigma_{2,2}(5-7+6+1-4) = \sigma_{2,2}(1)\] and \[\phi(4, 1, 9, 1)(5,7) =  \sigma_{2,2}(5+7-4-9) = \sigma_{2,2}(-1).\] Applying equation \eqref{sigma_eqn}, we have $\sigma_{2,2}(-1) = 0$ and $\sigma_{2,2}(1) = \left\lfloor\left(\frac{1}{2}+1\right)^2\right\rfloor = 2.$ Therefore, $g(\be,\ga,\al) = 2-0 = 2.$

\end{ex}

\begin{prop}\label{prop:2-3-row1}
Let $\mu$ and $\nu$ be partitions of $n$, where $\mu = (\mu_1,\mu_2)$ and $\nu = (\nu_1,\nu_2,\nu_3)$, such that $\mu_1 \le \nu_1$. Then the Kronecker product $s_\mu \ast s_\nu$ has a saturated Newton polytope.
\end{prop}
\begin{proof}
We have that $\max\{\lambda_1 \mid g(\mu,\nu,\lambda) \neq 0\} = |\mu \cap \nu| = \mu_1+\nu_2$, a result due to~\cite{Dvi93}. Let $\lambda = (\mu_1+\nu_2,n-\mu_1-\nu_2).$ We apply Theorem \ref{thm:kron-coeff} with $\al=(\nu_1,\nu_2,\nu_3), \be= (\mu_1,\mu_2)$. Then $a = \nu_3, b= \nu_2-\nu_3, c = \min(\nu_1-\nu_2, \nu_3)=\nu_3, d = |\nu_1-\nu_2-\nu_3| = \nu_1-\nu_2-\nu_3$, and 
\[g(\mu,\nu,\lambda) = (\phi(\nu_3,\nu_2-\nu_3, \nu_2+1,\nu_3)-\phi(\nu_3,\nu_2-\nu_3,\nu_1+2,\nu_3)(\lambda_2,\mu_2+1).\]
Since $\la_2 \le \mu_2+1 \le \nu_1+1$, 
\begin{align*}
    \phi(\nu_3,\nu_2-\nu_3,\nu_1+2,\nu_3)(\lambda_2,\mu_2+1) &= \sigma_{\nu_2-\nu_3+1,\nu_3+1}(\lambda_2+\mu_2+1-\nu_3-\nu_1-2)\\
    &= \sigma_{\nu_2-\nu_3+1,\nu_3+1}(\mu_2-\mu_1-1)\\
    &= 0.
\end{align*}
It follows that \[g(\mu,\nu,\lambda) = \phi(\nu_3,\nu_2-\nu_3, \nu_2+1,\nu_3)(\lambda_2,\mu_2+1).\]

Since $\mu_2+1 \ge \nu_2+\nu_3+1$,
\begin{align*}
    \phi(\nu_3,\nu_2-\nu_3, \nu_2+1,\nu_3)(\lambda_2,\mu_2+1) &= \sigma_{\nu_2-\nu_3+1, \nu_3+1}(\lambda_2-\mu_2-1+\nu_2+1+\nu_3-\nu_3)\\
    &=\sigma_{\nu_2-\nu_3+1, \nu_3+1}(0) = 1.
\end{align*} Then, among all partitions $\lambda \vdash n$ such that $g(\mu,\nu,\lambda) > 0$, $(n-\mu_1-\nu_2, \mu_1+\nu_2)$ is the unique maximal in dominance order, and therefore, by Lemma~\ref{lem:unique_dom}, $s_\mu \ast s_\nu$ has a SNP.
\end{proof}

\begin{rem}
We cannot expect to have unique maximal terms in general. 
 For instance, $s_{(5,4,4)}*s_{(7,6)} = s_{(3, 2, 2, 2, 2, 2)} + s_{(3, 3, 2, 2, 2, 1)} + s_{(3, 3, 3, 2, 1, 1)} + s_{(3, 3, 3, 2, 2)} + s_{(3, 3, 3, 3, 1)} + s_{(4, 2, 2, 2, 2, 1)} + 2s_{(4, 3, 2, 2, 1, 1)} + 2s_{(4, 3, 2, 2, 2)} + s_{(4, 3, 3, 1, 1, 1)} + 3s_{(4, 3, 3, 2, 1)} + s_{(4, 3, 3, 3)} + s_{(4, 4, 2, 1, 1, 1)} + 2s_{(4, 4, 2, 2, 1)} + 2s_{(4, 4, 3, 1, 1)} + 2s_{(4, 4, 3, 2)} + s_{(4, 4, 4, 1)} + s_{(5, 2, 2, 2, 1, 1)} + s_{(5, 2, 2, 2, 2)} + 2s_{(5, 3, 2, 1, 1, 1)} + 4s_{(5, 3, 2, 2, 1)} + 3s_{(5, 3, 3, 1, 1)} + 2s_{(5, 3, 3, 2)} + s_{(5, 4, 1, 1, 1, 1)} + 3s_{(5, 4, 2, 1, 1)} + 3s_{(5, 4, 2, 2)} + 4s_{(5, 4, 3, 1)} + s_{(5, 4, 4)} + s_{(5, 5, 1, 1, 1)} + 2s_{(5, 5, 2, 1)} + s_{(5, 5, 3)} + s_{(6, 2, 2, 1, 1, 1)} + s_{(6, 2, 2, 2, 1)} + s_{(6, 3, 1, 1, 1, 1)} + 4s_{(6, 3, 2, 1, 1)} + 3s_{(6, 3, 2, 2)} + 3*s_{(6, 3, 3, 1)} + 2s_{(6, 4, 1, 1, 1)} + 5s_{(6, 4, 2, 1)} + 2s_{(6, 4, 3)} + 2s_{(6, 5, 1, 1)} + 2s_{(6, 5, 2)} + s_{(6, 6, 1)} + s_{(7, 2, 2, 1, 1)} + s_{(7, 2, 2, 2)} + 2s_{(7, 3, 1, 1, 1)} + 4s_{(7, 3, 2, 1)} + s_{(7, 3, 3)} + 3s_{(7, 4, 1, 1)} + 3s_{(7, 4, 2)} + 2s_{(7, 5, 1)} + s_{(7, 6)} + s_{(8, 2, 2, 1)} + 2s_{(8, 3, 1, 1)} + 2s_{(8, 3, 2)} + 2s_{(8, 4, 1)} + s_{(8, 5)} + s_{(9, 2, 2)} + s_{(9, 3, 1)}.$ In this product, $(8, 5)$ and $(9,3,1)$ are incomparable maximal elements with respect to the dominance order. 
 
 Also, 
$s_{(6,6)}\ast s_{(8,2,1,1)} = s_{(4, 4, 2, 1,1)} + s_{(4, 4, 3, 1)} + s_{(5, 3, 1, 1, 1, 1)} + s_{(5, 3, 2, 1, 1)} + s_{(5, 3, 2, 2)} + s_{(5, 3, 3, 1)} + 2s_{(5, 4, 1, 1, 1)} + 3s_{(5, 4, 2, 1)} + s_{(5, 4, 3)} + s_{(5, 5, 1, 1)} + 2s_{(5, 5, 2)} + s_{(6, 2, 2, 1, 1)} + 2s_{(6, 3, 1, 1, 1)} + 3s_{(6, 3, 2, 1)} + s_{(6, 3, 3)} + 4s_{(6, 4, 1, 1)} + 2s_{(6, 4, 2)} + 2s_{(6, 5, 1)} + s_{(7, 2, 1, 1, 1)} + s_{(7, 2, 2, 1)} + 2s_{(7, 3, 1, 1)} + 2s_{(7, 3, 2)} + 2s_{(7, 4, 1)} + s_{(7, 5)} + s_{(8, 2, 1, 1)} + s_{(8, 3, 1)}.$ In  this product, $(7,5)$ and $(8,3,1)$ are incomparable maximal.
\end{rem}

\section{Multi-LR coefficients and Horn inequalities}\label{s:LRs}

\subsection{Monomial expansion via multi-LR coefficients}
As we observed, the Kronecker product does not necessarily have a unique dominating term $s_\la$. Furthermore, there are no known positive formulas available for many other cases. We thus move directly towards the monomial expansion. The coefficient of $x^\bba$, where $\bba =(a_1,a_2,\ldots)$ in $s_\mu \ast s_\nu$ can be determined as follows. From equation~\eqref{eq:tripple_cauchy} we have that
$$s_\mu \ast s_\nu(x) = \sum_\al g(\la,\mu,\nu) s_\la(x) = \langle \prod \frac{1}{1-x_iy_jz_k}, s_\mu(y) s_\nu(z)\rangle$$
over the rings $\Lambda_y$ and $\Lambda_z$. We have that 
$$\prod \frac{1}{1- u_i v_j} = \sum_\al m_\al(u) h_\al(v),$$
so substituting this with $u=x$ and $v = [y\cdot z] = (y_1z_1, y_1z_2,\ldots)$ we get that the coefficient at $m_\bba(x)$ above is $\langle s_\mu(y) \ast s_\nu (z) , h_\bba[y\cdot z] \rangle$. Observe that $h_\al = h_{\al_1} h_{\al_2} \cdots $ and that 
$$h_m[yz] = s_{(m)}[yz] = \sum_{\al \vdash m} s_\al(y)s_\al(z)$$ 
from the fact that $g((m), \la,\mu)=\delta_{\la,\mu}$ as $\mathbb{S}_{(m)}$ is the trivial representation. Thus we can expand the above as
\begin{equation}\label{eq:mon_LR}
\langle s_\mu(y) \ast s_\nu (z) , h_\bba[y\cdot z] \rangle =  \langle s_\mu(y) \ast s_\nu (z), \prod_i \sum_{\al^i \vdash a_i} s_{\al^i}(y) s_{\al^i}(z) \rangle
= \sum_{ \al^i \vdash a_i, i=1,...} c^{\mu}_{\al^1\al^2\cdots} c^{\nu}_{\al^1\al^2 \cdots}
\end{equation}

We now define the following set of points given by the concatenation of the vectors $\al^1,\al^2,\ldots,\al^k$:
\begin{align}\label{eq:p_mu_a}
P(\mu; \bba):= \{ (\al^1, \al^2,\cdots,\al^k) \in \mathbb{Z}_{\geq 0}^{\ell(\mu)k} :  c^{\mu}_{\al^1\al^2\cdots}>0 \text{ and } |a^i|=a_i \text{ for all }i=1,\ldots,k\}.
\end{align}
 $P(\mu;\bba) \neq \emptyset$ for all $\mu, \bba$ of the same size. This can be seen either by a greedy algorithm to construct $\al^1,\ldots$ giving a nonzero multi-LR coefficient, or by observing that $s_\mu \ast s_\mu = s_{(n)} + \cdots$ and contains every monomial of degree $n$, so for every $\bba$ there are some $\al^i \vdash a_i$ with $c^{\mu}_{\al^1\cdots} >0$. 

The monomials appearing in $s_\mu \ast s_\nu$  correspond to $\bba$, for which there exist $\al^1,\cdots$ with $c^\mu_{\al^1\cdots}>0$ and $c^{\nu}_{\al^1\cdots}>0$. Thus
\begin{prop}\label{prop:kron_monomials}
The set of monomial degrees $\bba=(a_1,\ldots,a_k)$ appearing in $s_\mu \ast s_\nu$ is given as
$$M_k(s_\mu \ast s_\nu) = \{ \bba \in \mathbb{Z}^k_{\geq 0}: P(\mu;\bba) \cap P(\nu;\bba) \neq \emptyset\}.$$ 
\end{prop}
We turn towards understanding the above set of points, and in particular, whether they  form a convex polytope.

\subsection{Horn inequalities for multi-LR's}

We first reduce our multi-LR positivity problem from~\eqref{eq:mon_LR} and~\eqref{eq:p_mu_a} to the case of regular LR coefficients.
Let again $c^{\mu}_{\alpha^1,\alpha^2,\ldots} = \langle s_{\alpha^1} s_{\alpha^2} \cdots, s_\mu \rangle$ be  the multi-LR coefficients.

\begin{thm}[\cite{van2001littlewood}]\label{thm:vlee}
Let $\lambda,\mu,\nu$ be partitions such that $|\lambda| = |\mu| + |\nu|$.
Then $c^{\lambda}_{\mu,\nu} = \langle s_\lambda, s_{\mu\diamond \nu}\rangle$,
where $\mu \diamond \nu$ denotes the skew shape $(\nu_1^{\ell(\mu)}+\mu,\nu)/\nu$.
\end{thm}

The skew shape $\mu \diamond \nu$ is just the union of two disjoint shapes $\mu$ and $\nu$ with the top right corner of $\nu$ touching the bottom left corner of $\mu$. 
We can generalize this as follows

\begin{lemma}\label{lem:multiLR-ordinary}
Let $\la \vdash n$. For a $k$-tuple of partitions $\al^1,\cdots, \al^k$ with $\ell(\al^i) \leq \ell$, such that $|\al^1|+\cdots+|\al^k|=n$ we have that
$c^{\la}_{\al^1\cdots \al^k} = \langle s_\la, s_{\al^1 \diamond \al^2 \diamond \cdots \diamond \al^k} \rangle = c^{\omega(\al)}_{\la, \delta_k(n,\ell)}$,
where $\al^1 \diamond \al^2 \diamond \al^3 \cdots = \al^1 \diamond (\al^2 \cdots)$ recursively,
$$ \omega(\al) := ( (n(k-1))^{\ell} + \al^1, (n(k-2))^\ell + \al^2, \cdots ,\al^k ) \text{ and } \delta_k(n,\ell) :=  \left( (n(k-1))^{\ell} , (n(k-2))^\ell , \cdots ,n^\ell \right).$$ 
\end{lemma}

\begin{proof}
We can recursively apply Theorem~\ref{thm:vlee}. Alternatively, we can see from the combinatorial definition via SSYTs that\[ s_{\alpha^1} \cdots s_{\alpha^k} = s_{\alpha^1 \diamond \alpha^2 \diamond \alpha^3 \cdots} = s_{\omega(\alpha) / \delta_k(n, \ell)}.\]
This holds because the skew shape consists of disjoint straight shapes \(\alpha^1, \alpha^2, \ldots, \alpha^k\).
\end{proof}

We next turn to LR positivity as described by the Horn inequalities. 
For a subset $I = \{i_1< i_2 < \dots < i_s\} \subset [r]$, let $\rho(I)$ denote the partition $\rho(I) := (i_s-s,\dots,i_2-2,i_1-1).$ We say a triple of subsets $I, J, K \subset [r]$ is \defn{LR-consistent} if they have the same cardinality $s$ and $c^{\rho(I)}_{\rho(J),\rho(K)} =1.$

\begin{thm}[\cite{zelevinsky1997littlewoodrichardson,klyachko,KT}]\label{thm:horn}
    Let $\lambda,\mu,\nu \in \mathbb{N}^r$ with weakly decreasing component. Then $c^\lambda_{\mu,\nu} >0$ if and only if $|\lambda| = |\mu| + |\nu|$ and $$\sum_{i\in I}\lambda_i \le \sum_{j\in J}\mu_j+\sum_{k\in K}\nu_k$$ for all $LR$-consistent triples $I,J,K \subset [r]$.
\end{thm}
A complete list of LR-consistent triples and the resulting inequalities for $r=6$ is given in Appendix~\ref{appendix_horn}.

For a set $I \subset \{1,\ldots, \ell k\}$ construct the set $D(I):= \{ (i,j) \in [k]\times[\ell], \text{ such that } \ell(i-1)+j \in I\}$, that is the set of pairs $( \lceil \frac{x}{\ell} \rceil , x\% \ell),$ where $x \in I$ and $x \% \ell$ is its remainder by division by $\ell$, adjusted to be in the range from $1$ to $\ell$.
Applying Theorem~\ref{thm:horn} with $\la = \omega(\al)$, $\mu$ and $\nu = \delta_k(n,\ell)$ from Lemma~\ref{lem:multiLR-ordinary}, and observing that
if $m=\ell(i-1) + j$ then $\omega(\al)_m = n(k-i) +\al^i_j$ and $(\delta_k(n,\ell))_m = n(k-i)$ we get the following.

\begin{Cor}\label{cor:polytope_horn}
Let $\ell(\mu)=\ell$ and $\bba=(a_1,\ldots,a_k)$. Then $P(\mu;\bba)$ is a polytope consisting of the points $(\al^1,\ldots,\al^k) \in \mathbb{Z}_{\geq 0}^{\ell k}$ satisfying the following linear conditions. 
\begin{align}
    \sum_j \al^i_j &\quad = \quad a_i, \qquad \text{ for }i\in[k];\\
    \al^i_j &\quad \geq \quad \al^i_{j+1},  \qquad \text{ for }j\in[\ell-1], \, i\in[k]; \\
    \sum_{ (i,j) \in D(I) } \left( n(k-i) +\al^i_j \right) &\quad \leq \quad \sum_{j \in J} \mu_j + \sum_{ (d,r) \in D(K)}  n(k-d) ,  
\end{align}
where the last inequalities hold for all LR-consistent triples $I,J,K \in[\ell k]$
\end{Cor}

\subsection{The case $k=3$.}
As we know the values of LR coefficients for the triples of partitions $\rho(I), \rho(J), \rho(K)$ when $|I|\leq 6$, see Appendix~\ref{appendix_horn}, we can write all the linear inequalities defining the set of $(\la,\mu,\nu)$ with $\ell(\la),\ell(\mu),\ell(\nu)\leq 6$ and see that they are the integer points in a convex polytope. In general, the polytope $P(\mu;\bba)$ is defined by a large set of inequalities, which are complicated and themselves recursively defined. It is not known whether it always has at least one integral nonzero vertex, see the Appendix. We will approach the first cases beyond Section~\ref{s:2-3-row}. 

We will restrict ourselves to the Kronecker product of a two-row and a three-row partition and monomials $x_1^{a_1}x_2^{a_2}x_3^{a_3}$. Let $\ell(\nu) =3$ and $\ell(\mu)=2$.
Our goal is to describe $P(\nu;a_1,a_2,a_3) \cap P(\mu;a_1,a_2,a_3)$. Since we need $c^{\mu}_{\al^1, \al^2,\al^3}>0$, which would imply that $\al^i \subset \mu$, we must have $\ell(\al^i) \leq 2$. 
Applying Lemma~\ref{lem:multiLR-ordinary}, we have that $c^\nu_{\al^1,\al^2,\al^3}= c^{\omega(\al) }_{\nu, \delta_3(n,2)}$ with 
$\omega(\al) = (2n+\alpha^1_1,2n+\alpha^1_2, n+\alpha^2_1, n+\alpha^2_2, \alpha^3_1,\alpha^3_2)$, $\nu = (\nu_1,\nu_2,\nu_3,0,0,0)$, $\delta_3(n,2) = (2n,2n,n,n,0,0)$. Since $n = |\nu| = |\alpha^1|+|\alpha^2|+|\alpha^3|$, we can derive the following relations from Theorem~\ref{thm:horn} applied to the triple of partitions $\omega(\al), \nu, \delta_3(n,2)$, to get that
\begin{align}\label{eq:c_nu_pos}
c^\nu_{\alpha^1,\alpha^2,\alpha^3}>0 \text{ if and only if } \qquad \qquad \qquad \qquad\qquad \qquad & \\
\notag  \max \{\alpha^1_1,\alpha^2_1,\alpha^3_1, \alpha^1_2+\alpha^2_2, \alpha^1_2+\alpha^3_2,\alpha^2_2+\alpha^3_2\} &\le \nu_1\\
    \notag \max \{\alpha^1_2,\alpha^2_2,\alpha^3_2\}  &\le \nu_2\\
    \notag \alpha^1_2+\alpha^2_2+\alpha^3_2 &\le \nu_2+\nu_3\\
  \notag   \max \{\alpha^1_1+\alpha^2_2+\alpha^3_2,\alpha^1_2+\alpha^2_1+\alpha^3_2,\alpha^1_2+\alpha^2_2+\alpha^3_1 \} &\le \nu_1+\nu_3\\
    \notag \max \{\alpha^1_1+\alpha^2_1+\alpha^3_2,\alpha^1_2+\alpha^2_1+\alpha^3_1,\alpha^1_1+\alpha^2_2+\alpha^3_1 \} &\le \nu_1+ \nu_2\\
    \notag \max \{\alpha^1_1+\alpha^1_2+\alpha^2_2+\alpha^3_2,\alpha^1_2+\alpha^2_1+\alpha^2_2+\alpha^3_2,\alpha^1_2+\alpha^2_2+\alpha^3_1 +\alpha^3_2\} &\le \nu_1+\nu_2.
\end{align}
Similarly we obtain the inequalities for $c^{\mu}_{\al^1,\al^2,\al^3}>0$, and observe that they are the same as \eqref{eq:c_nu_pos} with $(\nu_1,\nu_2,\nu_3)$ replaced by $(\mu_1,\mu_2,0)$. Noting that the inequalities are all of the form ``linear combination of $\al^i_j$'s $\leq$ linear combinations of $\nu_i$'s'', we can combine them for $\mu$ and $\nu$ as

\begin{align}\label{eq:c_nu_mu_pos}
c^\nu_{\alpha^1,\alpha^2,\alpha^3}c^\mu_{\al^1,\al^2,\al^3}>0 \text{ if and only if } \qquad \qquad \qquad \qquad\qquad & \\
\notag  \max \{\alpha^1_1,\alpha^2_1,\alpha^3_1, \alpha^1_2+\alpha^2_2, \alpha^1_2+\alpha^3_2,\alpha^2_2+\alpha^3_2\} &\le \min\{ \nu_1, \mu_1\}\\
    \notag \max \{\alpha^1_2,\alpha^2_2,\alpha^3_2\}  &\le \min\{ \nu_2, \mu_2\} \\
    \notag \alpha^1_2+\alpha^2_2+\alpha^3_2 &\le \min\{ \nu_2+\nu_3, \mu_2\} \\
  \notag   \max \{\alpha^1_1+\alpha^2_2+\alpha^3_2,\alpha^1_2+\alpha^2_1+\alpha^3_2,\alpha^1_2+\alpha^2_2+\alpha^3_1 \} &\le \min\{ \nu_1+\nu_3, \mu_1\} \\
    \notag \max \{\alpha^1_1+\alpha^2_1+\alpha^3_2,\alpha^1_2+\alpha^2_1+\alpha^3_1,\alpha^1_1+\alpha^2_2+\alpha^3_1 \} &\le \min\{ \nu_1+ \nu_2, \mu_1+\mu_2\} \\
    \notag \max \{\alpha^1_1+\alpha^1_2+\alpha^2_2+\alpha^3_2,\alpha^1_2+\alpha^2_1+\alpha^2_2+\alpha^3_2,\alpha^1_2+\alpha^2_2+\alpha^3_1 +\alpha^3_2\} &\le \min\{ \nu_1+\nu_2, \mu_1+\mu_2\}.
\end{align}

Since $\mu_1+\mu_2=n$, thus on the RHS we have some simplifications: $\nu_1+\nu_2 \leq \mu_1+\mu_2$. Note also that the case $\mu_1\leq \nu_1$ was resolved in Proposition~\ref{prop:2-3-row1}, and so we can assume that $\mu_1 > \nu_1$ and then $\nu_2+\nu_3=n-\nu_1 > n-\mu_1=\mu_2$. We also have that the $\max\{\al^1_2, \al^2_2,\al^3_2\} \leq \al^1_2 + \al^2_2 + \al^3_2$, and so we can replace the RHS of the second inequality by $\nu_2$, as the third inequality would imply it is $\leq \mu_2$. We can thus rewrite the above inequalities as

\begin{align}\label{eq:c_nu_mu_pos_2}
c^\nu_{\alpha^1,\alpha^2,\alpha^3}c^\mu_{\al^1,\al^2,\al^3}>0 \text{ if and only if } \qquad \qquad \qquad \qquad\qquad & \\
\notag  \max \{\alpha^1_1,\alpha^2_1,\alpha^3_1, \alpha^1_2+\alpha^2_2, \alpha^1_2+\alpha^3_2,\alpha^2_2+\alpha^3_2\} &\le  \nu_1\\
    \notag \max \{\alpha^1_2,\alpha^2_2,\alpha^3_2\}  &\le \nu_2 \\
    \notag \alpha^1_2+\alpha^2_2+\alpha^3_2 &\le  \mu_2 \\
  \notag   \max \{\alpha^1_1+\alpha^2_2+\alpha^3_2,\alpha^1_2+\alpha^2_1+\alpha^3_2,\alpha^1_2+\alpha^2_2+\alpha^3_1 \} &\le \min\{ \nu_1+\nu_3, \mu_1\} \\
    \notag \max \{\alpha^1_1+\alpha^2_1+\alpha^3_2,\alpha^1_2+\alpha^2_1+\alpha^3_1,\alpha^1_1+\alpha^2_2+\alpha^3_1 \} &\le  \nu_1+ \nu_2\\
    \notag \max \{\alpha^1_1+\alpha^1_2+\alpha^2_2+\alpha^3_2,\alpha^1_2+\alpha^2_1+\alpha^2_2+\alpha^3_2,\alpha^1_2+\alpha^2_2+\alpha^3_1 +\alpha^3_2\} &\le \nu_1+\nu_2.
\end{align}

\subsection{The polytope $P(\mu;\bba) \cap P(\nu;\bba)$ }

The linear inequalities~\eqref{eq:c_nu_mu_pos_2} describe a polytope in $\mathbb{R}^6$ for the variables $(\al^1_1,\al^1_2,\ldots)$. By Section~\ref{s:LRs} a monomial $x^\bba$ occurs in $s_\mu \ast s_\nu$  if and only if the set $P(\mu;\bba) \cap P(\nu;\bba)$ has a nonzero integer point. This set corresponds to the section of the polytope in~\eqref{eq:c_nu_mu_pos_2} with $\al^i_1+\al^i_2=a_i$ for $i=1,2,3$, as well as $\al^i_1 \geq \al^i_2$, which comes from  $\al^i$s being partitions. To simplify notation, let $x:=\al^1_1, y:=\al^2_1, z:=\al^3_1$, so $\al^1_2=a_1-x$ etc. Define $\mathcal{P}(\mu,\nu,\bba)$ to be that polytope, substituting the new constraints in~\eqref{eq:c_nu_mu_pos_2}; it is defined by the following inequalities
\begin{align}
\notag \mathcal{P}(\mu,\nu,\bba): =\Bigg\{(x,y,z) \in \mathbb{R}^3 \text{ such that } \. \qquad \qquad & \\
a_1 - \min(\nu_2, \mu_2, \frac{a_1}{2}) &\le x \le \min(a_1,\nu_1)\label{eqn1}
\tag{1}\\
a_2-\min(\nu_2, \mu_2, \frac{a_2}{2}) &\le y \le \min(a_2,\nu_1) \label{eqn2}
\tag{2}\\
a_3- \min(\nu_2, \mu_2,\frac{a_3}{2}) &\le z \le \min(a_3,\nu_1) \label{eqn3}
\tag{3} \\
\max(\nu_3, a_1+a_2-\nu_1) &\le x+y \label{eqn4}
\tag{4}\\
\max(\nu_3, a_1+a_3-\nu_1) &\le x+z \label{eqn5}
\tag{5}\\
\max(\nu_3, a_2+a_3-\nu_1) &\le y+z  \label{eqn6}
\tag{6}\\
\mu_1 &\le x+y+z \label{eqn7}
\tag{7}\\
\max(\nu_2,\mu_2) -a_1 &\le -x+y+z \le \nu_1+\nu_2-a_1\label{eqn8}
\tag{8}\\
\max(\nu_2,\mu_2) -a_2 &\le x-y+z \le \nu_1+\nu_2-a_2 \label{eqn9}
\tag{9}\\
\. \max(\nu_2,\mu_2) -a_3 &\le x+y-z \le \nu_1+\nu_2-a_3 \Bigg\} \label{eqn10}
\tag{10}
\end{align}

We can summarize these descriptions and derivation in the following. 

\begin{prop}\label{prop:2-3row2}
The monomial $x^\bba$ occurs in $s_\mu\ast s_\nu$ if and only if $P(\mu;\bba) \cap P(\nu;\bba) \neq \emptyset$. When $\ell(\mu)=2, \ell(\nu)=3$ and $\nu_1<\mu_1$ this is equivalent to $\mathcal{P}(\mu,\nu,\bba) \cap \mathbb{Z}^3 \neq \emptyset$.
\end{prop}

\section{Integer points in $\mathcal{P}(\mu,\nu,\bba)$}\label{s:int_points}

We are now ready to prove the counterpart of Proposition~\ref{prop:2-3-row1} by analyzing the polytope $\mathcal{P}(\mu,\nu,\bba)$. Suppose that the monomials $x^{\bba^i}$ for some $i=1,\ldots$ occur in $s_\mu\ast s_\nu(x_1,x_2,x_3)$, and suppose that $\bbc:=t_1 \bba^1 + t_2 \bba^2 + \cdots\in \mathbb{Z}^3$ for some $t_i \in [0,1]$ with $t_1+t_2+\cdots=1$ is a convex combination. We want to show that $x^\bbc$ also occurs in that Kronecker product. This is equivalent to understanding when $\mathcal{P}(\mu,\nu,\bba)$ is nonempty and has integer points. 

\begin{prop}\label{prop:average_polytope}
Suppose that $\mathcal{P}(\mu,\nu, \bba^i) \neq \emptyset$  for some vectors $\bba^i$, $i=1,\ldots,4$ and $\bbc =\sum_i t_i \bba^i$ for some $t_i\in[0,1]$ with $t_1+t_2+t_3+t_4=1$. Then  $\mathcal{P}(\mu,\nu, \bbc)\neq \emptyset$.
\end{prop}

\begin{proof}
The inequalities defining $\mathcal{P}(\mu,\nu,\bba)$ can be written in the form $A [x,y,z]^T \leq \bbv$ for a $31 \times 3$ matrix $A$ with entries $\{0,1,-1\}$ and vector $\bbv = B_1 [\mu_1,\mu_2]^T + B_2 [\nu_1,\nu_2,\nu_3]^T + B_3 [a_1, a_2,a_3]^T$. Here we separate the inequalities involving \(\min\) and \(\max\) functions. Specifically, the inequality \(\max\{p, q\} \leq x+y-z\) is equivalent to the two inequalities \(p \leq x+y-z\) and \(q \leq x+y-z\), and similarly for others.
 The bounds then become linear combinations of the parameters $\mu, \nu, \bba$ and can be written as a vector $\bbv$ above. 

Suppose now that $\mathcal{P}(\mu,\nu,\bba^i) \neq \emptyset$ for all $i$ and let $p_i:=(x^i,y^i,z^i) \in \mathcal{P}(\mu,\nu,\bba^i)$ be a collection of points in each, so $A p_i^T \leq B_1\mu^T + B_2 \nu^T + B_3 (\bba^i)^T$ for all $i$. Let $p:=\sum_i t_i p_i$. 
Then, since $t_i \geq 0$ and $\sum t_i =1$ we have
$$A p^T = \sum_i t_i Ap_i^T \leq \sum_i t_i ( B_1\mu^T + B_2 \nu^T + B_3 (\bba^i)^T) =B_1 \mu^T + B_2 \nu^T + B_3 (\sum_i t_i\bba^i)^T$$
So $p$ satisfies the inequalities for $\mathcal{P}(\mu,\nu,\bbc)$ and hence this polytope is nonempty.
\end{proof}

We will now show that this polytope is nonempty if and only if it has an integer point. 
\begin{thm}\label{thm:int_point}
If $\mathcal{P}(\mu,\nu,\bba) \neq \emptyset$ then it has an integer point, i.e. $\mathcal{P}(\mu,\nu,\bba)\cap \mathbb{Z}^3 \neq \emptyset$.
\end{thm}

This will follow from the next two propositions which further characterize the points in the polytope.

\begin{prop}\label{prop:int_point}
Suppose that $\mathcal{P}(\mu,\nu;\bba) \neq \emptyset$. Then it has a half-integer point, i.e. $\mathcal{P}(\mu,\nu;\bba) \cap (\frac12 \mathbb{Z} )^3 \neq \emptyset$.
\end{prop}

\begin{proof}
The statement would hold if we could always guarantee the existence of a half-integer vertex. The existence of such vertices is not apparent. However, we can find points near any vertex which are half-integers. Since we assume the polytope is nonempty it would have at least one vertex, and by our proof a half-integer point near it. We proceed by first characterizing the vertices in the following claim. 

\begin{claim}\label{claim:13or14integer} If $(x,y,z)$ is a vertex of $\mathcal{P}(\mu,\nu,\bba)$, then $x,y,z \in \frac14 \mathbb{Z}$ or $x,y,z \in \frac13 \mathbb{Z}$. 
\end{claim}

\begin{proof}[Proof of claim:] The vertices correspond to extremal points, i.e. when 3 or more of the defining inequalities become equalities. If $(x,y,z)$ is a vertex then it would be a solution to an equation of the general form
 $$A \begin{bmatrix} x \\ y \\ z \end{bmatrix} = \begin{bmatrix} v_1 \\ v_2 \\ v_3 \end{bmatrix}$$
for an invertible  $3 \times 3$ matrix $A$ with entries in $\{-1,0,1\}$, which come from the form of the listed inequalities. The values $v_1,v_2,v_3$ are from the set of bounds on the left or right of inequalities \eqref{eqn1}--\eqref{eqn10}, which are almost all integers except possibly for the LHS of the first 3 equations, which could be half integers. The vertex coordinates are thus given by $A^{-1} [v_1,v_2, v_3]^T$. We consider the different types of matrices $A$. The following cases give us all possibilities up to  sign changes of rows and row/column permutations. The cases are obtained by considering how many inequalities from each of the 3 groups group I:\eqref{eqn1}-\eqref{eqn3}, II:\eqref{eqn4}-\eqref{eqn6}, III:\eqref{eqn7}-\eqref{eqn10} we have. Since the inequalities are symmetric with respect to $x,y,z$, we can make some assumptions on the form of the chosen equations. For example, if we choose at least one equation from group I, then we can chose one of them to be $x=*$, and depending on whether we choose 1, 2 or 0 equations from group II and one or two more from group I we get cases~\ref{case1} and ~\ref{case2} below. Observe that each of these cases encompasses several possibilities of group I, II or III choices.  Next, if we do not choose an equation from group I, we can have either two equations from group II (giving case~\ref{case3}), one equation from group II (case~\ref{case4}). The last two cases come from the choices of equations only from group III. 

\begin{enumerate}[(a)]

\item\label{case1} $A = \begin{bmatrix} 1 & 0 & 0 \\ * & 1 & 0 \\ * & * & 1\end{bmatrix}$, where $* \in \{0,+1,-1\}$. In that case $(x,y,z)$ is an integral linear combination of the $v_1,v_2,v_3$ and thus is in $\frac12 \mathbb{Z}$. 

\item\label{case2} $A = \begin{bmatrix} 1 & 0 & 0 \\ * & 1 & 1 \\ * & 1& -1\end{bmatrix}$, where $* \in \{0,+1,-1\}$. Observe that other $\pm 1$ combinations in the second and third row would either be equivalent up to sign change or result in $\det A=0$.  In this case, if $x \in \mathbb{Z}$ we have $y,z \in \frac12 \mathbb{Z}$, otherwise if $x =\frac{a_1}{2}$ then $y,z \in \frac14 \mathbb{Z}$. 

\item\label{case3} $A = \begin{bmatrix} 1 & 1 & 0 \\ 1 & 0 & 1 \\  *  & \pm 1 & \pm 1\end{bmatrix}$ coming from at least two equations from \eqref{eqn4}-\eqref{eqn6} and one from \eqref{eqn7}-\eqref{eqn10}. Then $v_1,v_2,v_3 \in \mathbb{Z}$ and in most cases the solutions are integral linear combinations of $v_1, v_2, v_3$. In the case where the last line is $[-1,1,1]$ we can have $\frac13 \mathbb{Z}$. In the case when the last line is $[0,1,1]$ then the solutions are in $\frac12 \mathbb{Z}$.

\item\label{case4} $A = \begin{bmatrix} 1 & 1 & 0 \\ 1 & \pm 1 &  \pm 1 \\  1  & \pm 1 & \pm 1\end{bmatrix}$ where we have one equation from \eqref{eqn4}-\eqref{eqn6} and 2 equations from \eqref{eqn7}-\eqref{eqn10}. Then $v_i \in \mathbb{Z}$, and $[y,z]$ is a solution to $\begin{bmatrix} \epsilon_1-1 & \epsilon_2 \\ \epsilon_3-1 & \epsilon_4 \end{bmatrix} [y,z]^T = [ v_2-v_1, v_3-v_1]^T$, where $\epsilon_i \in \{-1,+1\}$. For this matrix to be invertible we must have $\epsilon_1, \epsilon_3$ not both equal to $1$. If $\epsilon_1 \neq \epsilon_3$, then the determinant is $\pm 2$ and $y,z \in \frac12 \mathbb{Z}$. If $\epsilon_1=\epsilon_3=-1$, then $z \in \frac12 \mathbb{Z}$ and $y \in \frac14 \mathbb{Z}$.

\item\label{case5} $ A =   \begin{bmatrix} 1 & 1 & 1 \\ 1 & 1 & -1 \\  1  & -1 & 1\end{bmatrix}$, taking one equation from \eqref{eqn7} and two from \eqref{eqn8}-\eqref{eqn10}. Then $x,y,z \in \frac12 \mathbb{Z}$. 

\item\label{case6} $A =   \begin{bmatrix} -1 & 1 & 1 \\ 1 & 1 & -1 \\  1  & -1 & 1\end{bmatrix}$, then $x,y,z \in \frac12 \mathbb{Z}$.

\end{enumerate}

\end{proof}

\begin{claim}
Suppose that $\mathcal{P}(\mu,\nu,\bba)$ has a vertex with $x,y,z \in \frac13 \mathbb{Z}$. Then it has a point in $\frac12 \mathbb{Z}$.
\end{claim}

\begin{proof}[Proof of claim:] From the proof of Claim \ref{claim:13or14integer}, we see that a nonintegral vertex in $\frac13 \mathbb{Z}$ can occur only from equations \eqref{eqn4}, \eqref{eqn5} and \eqref{eqn8} (up to permutation of $x,y,z$). Let $v_1 = \max\{ \nu_3, a_1+a_2 -\nu_1\}$, $v_2 = \max\{ \nu_3, a_1+a_3-\nu_1\}$ and $v_3 = \max\{\nu_2,\mu_2\} -a_1$ or $v_3 = \nu_1+\nu_2-a_1$. The vertex is either $(x,y,z) = (x' + \frac23, y'+\frac13, z'+\frac13) $ or $(x' + \frac13, y'+\frac23, z'+\frac23)$ for some $x',y',z'\in \mathbb{Z}$. 

In the first case, we have that $(x'+1,y',z')$ and $(x'+\frac 12, y'+\frac12, z'+\frac12)$ still satisfy inequalities \eqref{eqn1}-\eqref{eqn6} as the bounds are integers or half-integers. We have $(x'+1)+y'+z' = x+y+z -\frac13 \geq \mu_1$ since  $\lfloor x+y+z  \rfloor \geq \mu_1$ and similarly $(x'+\frac 12, y'+\frac12, z'+\frac12)$ satisfies inequality (7). Moreover, $-(x'+1) + y'+z' = -(x+1/3)+y-1/3+z-1/3 = -x+y+z -1 \leq \nu_1 + \nu_2 -a_1$,
$(x'+1)-y'+z' = (x+\frac13) -y-z =\lfloor x-y-z \rfloor$ and satisfies inequality \eqref{eqn9} as both sides are integers, similarly with inequality \eqref{eqn10}. If $(x'+1,y',z') \in \mathcal{P}(\mu,\nu,\bba)$, we are done. Now suppose that $(x'+1,y',z') \not \in \mathcal{P}(\mu,\nu,\bba)$. Then $-(x'+1)+y'+z' < \max\{\nu_2,\mu_2\} -a_1$, so $-x' +y'+z' = \max\{\nu_2,\mu_2\} -a_1$.
The point $(x'+\frac12, y'+\frac12, z'+\frac12)= (x,y,z) + (\frac16, -\frac16,-\frac16)$ still satisfies \eqref{eqn9} and \eqref{eqn10} since the bounds are integral and 
$(x'+\frac12) - (y'+\frac12) + (z'+\frac12) = x'-y'+z' +\frac12 = \lfloor x-y+z\rfloor + \frac12$. For inequality \eqref{eqn8}, we have
$$-(x'+\frac12) + (y'+\frac12) + (z'+\frac12) = -x'+y'+z' +\frac12 = -x+y+z +\frac12 =  \max\{\nu_2,\mu_2\} -a_1 + \frac12$$ 
according to the implication from the assumption that $(x'+1,y',z') \not \in \mathcal{P}(\mu,\nu,\bba)$. Since the upper and lower bounds of these inequalities are at least $\nu_1 + \nu_2 -\max\{\nu_1,\mu_2\} \geq \mu_1 - \nu_3 \geq n/6$ apart, we have that $(x'+\frac12,y'+\frac12,z'+\frac12) \in \mathcal{P}(\mu,\nu,\bba)$. 

In the second case, $(x,y,z) = (x' + \frac13, y'+\frac23, z'+\frac23)$ we consider $(x', y'+1,z'+1)$ and $(x'+\frac12, y'+\frac12, z'+\frac12)$ and performing an analogous analysis we see that one of these points is in $\mathcal{P}(\mu,\nu,\bba)$.
\end{proof}

\begin{claim} Suppose that $\mathcal{P}(\mu,\nu,\bba)$ has a vertex with $x,y,z \in \frac14 \mathbb{Z}$. Then it has a point in $\frac12 \mathbb{Z}$.
\end{claim}

\begin{proof}[Proof of claim:] From the proof of Claim 1, we observe that a nonintegral vertex in $\frac14 \mathbb{Z}$ can occur from case~\ref{case2} or ~\ref{case4}. In case~\ref{case2}, a nonintegral vertex in $\frac14 \mathbb{Z}$ can occur only when $x = \frac{a_1}{2}$ from equations \eqref{eqn1}, \eqref{eqn6} and \eqref{eqn9} (up to permutation of $x,y,z$). The vertex can be expressed as either $(x,y,z) = (x'+\frac{1}{2}, y'+\frac{1}{4}, z'+\frac{3}{4})$ or $(x,y,z) = (x'+\frac{1}{2}, y'+\frac{3}{4}, z'+\frac{1}{4})$ for some $x',y',z' \in \mathbb{Z}$. We have that $(x'+\frac{1}{2}, y'+\frac{1}{2},z'+1)$  and $(x'+\frac{1}{2}, y'+1,z'+\frac12)$  will satisfy all inequalities \eqref{eqn1}-\eqref{eqn10} in the two cases, respectively, as the bounds are either integers or half-integers.

In case~\ref{case4}, a nonintegral vertex in $\frac14 \mathbb{Z}$ can only result from equations \eqref{eqn4}, \eqref{eqn8} and \eqref{eqn9} (up to permutation of $x,y,z$). Such a vertex can be represented as either $(x,y,z) = (x'+\frac{1}{4}, y'+\frac{3}{4}, z'+\frac12)$ or $(x,y,z) = (x'+\frac{3}{4}, y'+\frac{1}{4}, z'+\frac{1}{2})$ for some $x',y',z' \in \mathbb{Z}$. Then we have that $(x'+\frac{1}{2}, y'+1,z'+\frac{1}{2})$  and $(x'+1, y'+\frac12,z'+\frac12)$  will satisfy all inequalities \eqref{eqn1}-\eqref{eqn10} in the two cases, respectively, as the bounds are integers or half-integers.
\end{proof}
In all other cases, the vertices are already in $\mathbb{Z}$ or $\frac12 \mathbb{Z}$, so we always have a half-integer point. 
\end{proof}

\begin{prop}\label{prop:int_point2}
Suppose that $\mathcal{P}(\mu,\nu,\bba) \cap (\frac12 \mathbb{Z})^3 \neq \emptyset$, then it has an integer point, i.e. $\mathcal{P}(\mu,\nu,\bba) \cap \mathbb{Z}^3 \neq \emptyset$.
\end{prop}

\begin{proof}
Suppose that $(u+\epsilon_1, v+\epsilon_2, w+\epsilon_3) \in \mathcal{P}(\mu,\nu,\bba)$ for some $\epsilon_i\in\{0,\frac{1}{2}\}$ and $u,v,w \in \mathbb{Z}$. We will show that $(u + \delta_1, v+\delta_2, w+\delta_3) \in \mathcal{P}(\mu,\nu,\bba)$ for some $\delta_i \in \{0,1\}.$

\begin{claim}\label{claim_odd_halfs}
If 1 or 3 components in $(\epsilon_1, \epsilon_2,\epsilon_3)$ are equal to $\frac{1}{2}$, then we have a point in $\mathcal{P}(\mu,\nu,\bba)$ with 0 or 2 of the components equal to $\frac{1}{2}$.
\end{claim}
\begin{proof}[Proof of Claim]
    Suppose that one component of $(\epsilon_1, \epsilon_2,\epsilon_3)$ is $\frac{1}{2}$, say $\epsilon_1=\frac12$.  Then, other than inequality~\eqref{eqn1}, the inequalities involving the variable $x$  are of the form $a \le m + \frac{1}{2} \le b$, where $a$, $b$, $m$ are integers, which hold if and only if both $a \le m \le b$ and $a\le m + 1 \le b$ hold. Moreover, inequality~\eqref{eqn1} holds for $u+1$ as the upper bound is an integer. It follows that if $(u,v,w) + (\frac12,0,0) \in \mathcal{P}(\mu,\nu,\bba)$ then $(u+1,v,w) \in \mathcal{P}(\mu,\nu,\bba).$ 

    If all three components in $(\epsilon_1, \epsilon_2,\epsilon_3)$ are equal to $\frac{1}{2}$, we can increase one value by $\frac12$, to get $(u+1,v+\frac12,w+\frac12)$. Such a point will not be in the polytope only if it no longer satisfies some of the inequalities \eqref{eqn8}-\eqref{eqn10}. However, the expressions are of the form $a \leq -x+y+z = m+\frac12 \leq b$ with $a,b,m\in \mathbb{Z}$, and so changing a value by $\frac12$ will keep the expressions within $[a,b]$. So $(u+1,w+\frac12, w +\frac12) \in \mathcal{P}(\mu,\nu,\bba)$. 
\end{proof}

\begin{claim}\label{claim_two_halfs} If $(x+ \frac{1}{2},y+\frac{1}{2},z) \in \mathcal{P}(\mu,\nu,\bba)$, then there exists $(\delta_1,\delta_2) \in \{0,1\}^2$ such that $(x + \delta_1, y+\delta_2, z) \in \mathcal{P}(\mu,\nu,\bba)$.
\end{claim}
\begin{proof}[Proof of Claim~\ref{claim_two_halfs}]
    
Let $(x+ \frac{1}{2},y+\frac{1}{2},z) \in \mathcal{P}(\mu,\nu,\bba)$, so it satisfies all inequalities~\eqref{eqn1}-\eqref{eqn10}. Suppose, to the contrary, that $(x + \delta_1, y+\delta_2, z) \notin \mathcal{P}(\mu,\nu,\bba)$ for every $(\delta_1,\delta_2) \in \{0,1\}^2.$

If $(x+1, y+1, z) \notin \mathcal{P}(\mu,\nu,\bba)$, then, since it satisfies all inequalities except possibly the upper bound in inequality~\eqref{eqn10}, we should have $x+1+y+1-z \nleq \nu_1+\nu_2-a_3$. Since  $x+\frac{1}{2}+y+\frac{1}{2}-z \leq \nu_1+\nu_2-a_3$, we have
\begin{equation} \label{eq:x1y1z}
    x+y-z = \nu_1+\nu_2-a_3-1.
\end{equation}

The point $(x+1,y,z)$ satisfies all inequalities \eqref{eqn1},\eqref{eqn3}-\eqref{eqn7}, \eqref{eqn10}, so if $(x+1, y, z) \notin \mathcal{P}$ one of the inequalities \eqref{eqn8} or \eqref{eqn9} is not satisfied. Then either $x+1-y+z \nleq \nu_1+\nu_2-a_2$ or  $-x-1-y+z \ngeq \max(\nu_2,\mu_2)-a_1$. Since $(x+\frac12, y+\frac12,z)$ satisfy these inequalities, this implies that one of the following equations holds: 
\begin{equation} \label{eq:x1yz-1}
    x-y+z = \nu_1+\nu_2-a_2,
\end{equation}
\begin{equation} \label{eq:x1yz-2}
    -x+y+z = \max(\nu_2,\mu_2)-a_1,
\end{equation}
\begin{equation} \label{eq:x1yz-3}
    y +\frac12= \frac{a_2}{2}
\end{equation}

Similarly, if $(x, y+1, z) \notin \mathcal{P}(\mu,\nu,\bba)$, then one of the following equations holds: \begin{equation} \label{eq:xy1z-1}
    x-y+z = \max(\nu_2,\mu_2)-a_2
\end{equation} \begin{equation} \label{eq:xy1z-2}
    -x+y+z = \nu_1+\nu_2-a_1,
\end{equation}
\begin{equation} \label{eq:xy1z-3}
    x +\frac12= \frac{a_1}{2}.
\end{equation}

If $(x,y,z)\notin \mathcal{P}(\mu,\nu,\bba)$, then $x+y\ngeq \max(\nu_3,a_1+a_2-\nu_1)$, $x+y+z \ngeq \nu_1$ or $x+y-z \ngeq \max(\nu_2,\mu_2)-a_3$, which implies that one of the following equations holds:
\begin{equation}\label{eq:xyz-1}
x+y = \max(\nu_3,a_1+a_2-\nu_1)-1,
\end{equation}
\begin{equation}\label{eq:xyz-2}
x+y-z = \max(\nu_2,\mu_2)-a_3-1,
\end{equation}
\begin{equation}\label{eq:xyz-3}
x+y+z = \mu_1-1,
\end{equation}
\begin{equation} \label{eq:xyz-4}
    x +\frac12= \frac{a_1}{2},
\end{equation}
\begin{equation} \label{eq:xyz-5}
    y+\frac12 = \frac{a_2}{2}.
\end{equation}

We now consider the possible combinations of the above equalities, at least one from each group, and will aim to reach a contradiction in each case. We have that equation \eqref{eq:x1y1z} must hold in every case.

{\bf Case 1}: At least one of $x+\frac12 = \frac{a_1}{2}$ or $y+\frac12 = \frac{a_2}{2}$ holds. We have the following subcases.

\begin{itemize}

\item[Case 1.1:] If $x+\frac12 = \frac{a_1}{2}$ and $y+\frac12 = \frac{a_2}{2}$ and equation \eqref{eq:x1y1z} holds, it follows that $z = \nu_3-\frac{a_1+a_2}{2}$, but then for equation~\eqref{eqn8} we would have $-(x+\frac12)+(y+\frac12)+z = \nu_3-a_1 \leq \max(\nu_2,\mu_2)-a_1.$ 
Hence we must have $\nu_3 = \max\{\nu_2,\mu_2\}$ and thus $\mu_2 \leq \nu_2=\nu_3$. Then the LHS of equation~\eqref{eqn3} gives 
$\frac{a_3}{2} \leq a_3 - \min\{\mu_2,\frac{a_3}{2}\} \leq z= \nu_3 - \frac{a_1+a_2}{2}$, so $\nu_3 \geq \frac{a_1+a_2+a_3}{2}=n/2$, which is not possible.

\item[Case 1.2:] If $x+\frac12 = \frac{a_1}{2}$, but $y+\frac12 > \frac{a_2}{2}$ and equation \eqref{eq:x1yz-1} holds, then $z = \nu_1+\nu_2-a_2+y-x > \nu_1+\nu_2-\frac{a_1+a_2}{2}$. It follows that $-x+y+z > \nu_1+\nu_2-a_1$, contradicting inequality \eqref{eqn8} for $(x+1/2,y+1/2,z) \in \mathcal{P}$. 

\item[Case 1.3:] If $x+\frac12 = \frac{a_1}{2}$ and $y+\frac12 > \frac{a_2}{2}$ and equation \eqref{eq:x1yz-2} holds, then $z < \max (\nu_2,\mu_2)-\frac{a_1+a_2}{2}$. However, it implies that $x-y+z < \max(\nu_2,\mu_2) - a_2 $, contradicting inequality \eqref{eqn9} for $(x+1/2,y+1/2,z)\in \mathcal{P}$. 

\item[Case 1.4:] If $x+\frac12 > \frac{a_1}{2}$ and $y+\frac12 =\frac{a_2}{2}$, by a completely analogously argument, similar contradictions are reached.
\end{itemize}

\smallskip

{\bf Case 2}: We can now assume that $x \geq \frac{a_1}{2}$ and $y \geq \frac{a_2}{2}$, so neither of equations \eqref{eq:x1yz-3},\eqref{eq:xy1z-3}, \eqref{eq:xyz-4}, and \eqref{eq:xyz-5} are satisfied.

Since $\max(\nu_2, \mu_2) < \nu_1 + \nu_2$, it is not possible for equations \eqref{eq:x1yz-1} and \eqref{eq:xy1z-1} to both be satisfied simultaneously. Similarly, equations \eqref{eq:x1yz-2} and \eqref{eq:xy1z-2}, \eqref{eq:x1y1z} and \eqref{eq:xyz-2} cannot both be satisfied simultaneously. Thus, it can be reduced to two possibilities:
\begin{itemize}
    \item[Case 2.1:] When equations \eqref{eq:x1y1z},  \eqref{eq:x1yz-1}, \eqref{eq:xy1z-2} hold, the system of equations has solution 
    \[\begin{cases}
        x = \nu_1+\nu_2-\frac{a_2+a_3+1}{2}\\
        y = \nu_1+\nu_2-\frac{a_1+a_3+1}{2}\\
        z = \nu_1+\nu_2-\frac{a_1+a_2}{2}\\
    \end{cases}.\]
    If equation \eqref{eq:xyz-1} holds, we have
    \[x+y = 2(\nu_1+\nu_2)-\frac{a_1+a_2}{2}-a_3-1=\max(\nu_3,a_1+a_2-\nu_1)-1.\]
    If $\nu_3 \ge a_1+a_2-\nu_1$, using $a_1+a_2+a_3=n=\nu_1+\nu_2+\nu_3$,  we have 
    \begin{align*}
        2(\nu_1+\nu_2)-\frac{a_1+a_2}{2}-a_3-1 &=\nu_3-1\\
        a_3 &= 3(\nu_1+\nu_2-\nu_3) > n,
    \end{align*}
    which is impossible.
    
    If $\nu_3 < a_1+a_2-\nu_1$, then we have 
    \begin{align*}
        2(\nu_1+\nu_2)-\frac{a_1+a_2}{2}-a_3-1 &=a_1+a_2-\nu_1-1\\
        a_1+a_2 &= 2(2\nu_1+\nu_2-\nu_3)> n,
    \end{align*}
    which is also impossible.

    If equation \eqref{eq:xyz-3} holds, we have
    \[x+y+z =3(\nu_1+\nu_2)-(a_1+a_2+a_3)-1=3(\nu_1+\nu_2)-n-1 = \mu_1-1,\]
    from which it follows $\mu_1 = 2(\nu_1+\nu_2)-\nu_3 \geq n$, contradicting our assumption.
    \item[Case 2.2:] When equations \eqref{eq:x1y1z},  \eqref{eq:x1yz-2}, \eqref{eq:xy1z-1} hold, the system of equations has solution 
    \[\begin{cases}
        x = \frac{\nu_1+\nu_2+\max(\nu_2,\mu_2)-a_2-a_3-1}{2}\\
        y = \frac{\nu_1+\nu_2+\max(\nu_2,\mu_2)-a_1-a_3-1}{2}\\
        z = \max(\nu_2,\mu_2)-\frac{a_1+a_2}{2}
    \end{cases}.\]
    By inequality \eqref{eqn3}, we have  $z = \max(\nu_2,\mu_2)-\frac{a_1+a_2}{2} \ge \frac{a_3}{2}$, which implies that $\max(\nu_2,\mu_2) \ge \frac{n}{2}$ and hence $\mu_1 = \mu_2 = \frac{n}{2}.$

    If equation \eqref{eq:xyz-1} holds, we have
    \[x+y = \nu_1+\nu_2+\max(\nu_2,\mu_2)-\frac{a_1+a_2}{2}-a_3-1=\max(\nu_3,a_1+a_2-\nu_1)-1,\]
    which gives that 
$\nu_1+\nu_2+\mu_2-\frac{n}{2}-\frac{a_3}{2} = \nu_1+\nu_2-\frac{a_3}{2} =  \max(\nu_3,a_1+a_2-\nu_1).
$

\indent If $a_1+a_2-\nu_1 \ge \nu_3$, then $n \geq a_1+a_2+\frac{a_3}{2} = 2\nu_1+\nu_2 \geq n$.
Then we must have $a_3=0$, $\nu_1=\nu_2=\nu_3=n/3$ and $\mu_1=\mu_2=n/2$. Then  $x+\frac12 = \frac12 (\frac76 n - a_2)$, $y+\frac12 = \frac12(\frac76 n -a_1)$ and $x+y+1-z = \frac16 n < \max\{\nu_2,\mu_2\}$ and this point would not satisfy inequality~\eqref{eqn10} and be in $\mathcal{P}(\mu,\nu,\bba)$, so this is also impossible.

\indent If $a_1+a_2-\nu_1 < \nu_3$, then we have $\nu_1+\nu_2-\frac{a_3}{2} = \nu_3$, which implies that $v_3 = \frac{n}{2}-\frac{a_3}{4}.$ Substituting in the values, we can rewrite the solution as  \[\begin{cases}
        x = \frac{1}{2}(a_1+\frac{a_3}{4}-1)\\
        y = \frac{1}{2}(a_2+\frac{a_3}{4}-1)\\
        z = \frac{a_3}{2}
    \end{cases}.\] 
    
    Since $\frac{n}{3} \ge \nu_3 = \frac{n}{2}-\frac{a_3}{4}$, we can derive that $\frac{3}{2}a_3\ge n$ and $a_1+a_2=n-a_3  \le \frac{a_3}{2}$. The upper bounds in inequalities \eqref{eqn1}, \eqref{eqn2} for the integers $x+1,y+1$ hold since they hold for $x+\frac12, y+\frac12$. Thus $x+1 \leq a_1$, implying $\frac{a_3}{4}+1 \leq a_1$, and similarly $\frac{a_3}{4} + 1 \leq a_2$. Then $a_1 +a_2 \geq \frac{a_3}{2}+2$ violating the above conclusion. 

    If equation \eqref{eq:xyz-3} holds, we have
    \[x+y+z =2\max(\nu_2,\mu_2)+\nu_1+\nu_2-n-1 = 2\mu_2 -\nu_3-1= \mu_1-1,\]
    which implies that $\mu_1 = \nu_3$, reaching another contradiction.
\end{itemize}

In all cases, we reached contradictions, so the claim is true and there is an integer point.
\end{proof}

From Claim~\ref{claim_odd_halfs} we have that a point in $\mathcal{P}(\mu,\nu,\bba)\cap (\frac12 \mathbb{Z})^3$ implies that there is either an integer point, or a point with one integer and two half-integer entries. Claim~\ref{claim_two_halfs} shows that then there is always an integer point.
\end{proof}
\begin{proof}[Proof of Theorem~\ref{thm:int_point}]
Propositions~\ref{prop:int_point} and~\ref{prop:int_point2} together show that if the polytope is nonempty then it has an integer point. 
\end{proof}

\begin{proof}[Proof of Theorem~\ref{thm:main}]
Let $x_1^{a^i_1}x_2^{a^i_2}x_3^{a^i_3}$, for $i=1,\ldots$ be monomials appearing in $s_\mu \ast s_\nu (x_1,x_2,x_3)$  with nonzero coefficients. By Proposition~\ref{prop:2-3row2} we have that $\mathcal{P}(\mu,\nu;\bba^i) \cap \mathbb{Z}^3 \neq \emptyset$. Suppose that $(c_1,c_2,c_3)$ is in the convex hull of $\{ \bba^i \}_i$, so $\bbc = \sum_i t_i \bba^i$ for some $t_i \in [0,1]$ with $t_1+t_2+\cdots=1$. By Proposition~\ref{prop:average_polytope} we have that $\mathcal{P}(\mu,\nu,\bbc) \neq \emptyset$. Then if $c_i \in \mathbb{Z}$ by Theorem~\ref{thm:int_point} we have $\mathcal{P}(\mu,\nu;\bbc) \cap \mathbb{Z}^3 \neq \emptyset$ and thus ${\bf x}^\bbc$ appears as a monomial in $s_\mu \ast s_\nu$. By the~\eqref{eq:snp} characterization then $s_\mu \ast s_\nu(x_1,x_2,x_3)$ has a saturated Newton polytope. 
\end{proof}

\section{Positivity of Kronecker coefficients}\label{s:pos}

First, we will discuss the limiting case of the SNP property. 

\begin{proof}[Proof of Theorem~\ref{thm:limit}]
By Caratheodory's theorem, it is enough to show that if every point which is a convex combination of $k+1$ points from our set  is contained in the set, then the set is convex.

    Suppose that $\al^1, \al^2, \cdots, \al^{k+1} \in \bigcup_{p=1}^\infty \frac{1}{p} M_k({p\la}, {p\mu})$, where we abbreviate the notation $M_k(p\la,p\mu):=M_k(s_{p\la}\ast s_{p\mu})$. Then $\al^i \in \frac{1}{p_i} M_k(p_i \la, p_i \mu)$ for some $p_i$. First, let $p = lcm(p_1,\ldots,p_k)$. Since $p_i \al^i \in M_k(p_i \la, p_i \mu)$ there must be a $\be^i \succ p_i \al^i$ with $g(p_i \la, p_i \mu, \be^i) >0$. By the semigroup property then $g( p\la, p\mu, \frac{p}{p_i} \be^i)>0$, so that Kronecker product contains $s_{p/p_i \be^i}$ and all of its monomials. Since $\gamma^i:= p/p_i \be^i \succ p/p_i (p_i \al^i) = p\al^i$ then it contains $\al^i$ and $\al^i \in \frac{1}{p} M_k(p\al, p\mu)$ for all $i$ with $p\al^i \prec \gamma^i$ and $g(p\la,p\mu, \gamma^i)>0$. 
    
    Suppose that $\theta \in \mathbb{Q}^k$ is a convex combination of $\al^1,\ldots,\al^{k+1}$, i.e. $\theta = \sum_i t_i \al^i$ for some $t_i \in [0,1]$ with $t_1+t_2 +\cdots =1$. By Caratheodory's theorem $k+1$ points suffice.  Moreover, we can assume that the numbers $t_i \in \mathbb{Q}$, because the vector $(t_1,\ldots,t_{k+1})$ is a solution (not necessarily unique) to a system of $k$ linear equations with rational coefficients coming from $\al^i$, $\theta$, and the equation $\sum_i t_i=1$. We can then write $t_i = \frac{q_i}{q}$ for $q_i \in \mathbb{Z}$ and the same $q\in \mathbb{Z}$.
    By the semigroup property then 
    $g(q_ip \la, q_i p \mu, q_i  q_i \gamma^i)>0$ for all $i$ and so, again by semigroup, we have
    $$g \left( (\sum_i p q_i) \la, (\sum_i q_i p) \mu, \sum_i q_i \gamma^i \right) >0,$$ 
    so $s_{\sum_i q_i \gamma^i}$ appears in $s_{pq\la} \ast s_{pq \mu}$.
    Since $qp \theta = \sum_i q_i p\al^i \prec \sum_i q_i \gamma^i$ in the dominance order, the monomial $qp\theta$ appears in that product and 
    $qp \theta \in M_k (qp \la, qp \mu)$, so $\theta \in \frac{1}{qp} M_k(qp\la,qp\mu)$, which completes the proof.  
\end{proof}

We next consider positivity criteria for Kronecker coefficients.

Suppose that $g(\la,\mu,\nu)>0$, then $s_\la$ appears in $s_\mu \ast s_\nu$, and so its leading monomial $m_\la$ also appears, so $\mathcal{P}(\mu,\nu,\la)\cap \mathbb{Z}^{r} \neq \emptyset$, where $r = \min\{ \ell(\mu),\ell(\nu) \} \ell(\la)$. Then from Section~\ref{s:LRs} we must have that $P(\mu;\la) \cap P(\nu;\la)$ has an integer point. We can then apply Corollary~\ref{cor:polytope_horn} and its inequalities to infer that the following polytope $\mathcal{P}(\mu,\nu;\la)$ has an integer point.
\begin{prop}
    Suppose that $g(\la,\mu,\nu)>0$. Then there exist nonnegative integers $\al^i_j$ satisfying
     \begin{align}\label{eq:all inequalities}
    \sum_j \al^i_j &\quad = \quad \la_i, \qquad \text{ for }i\in[k];\\
    \al^i_j &\quad \geq \quad \al^i_{j+1},  \qquad \text{ for }j\in[\ell-1], \, i\in[k]; \\
    \sum_{ (i,j) \in D(I) } \left( n(k-i) +\al^i_j \right) &\quad \leq \quad \sum_{j \in J} \mu_j + \sum_{ (d,r) \in D(K)}  n(k-d) ,  \\
    \sum_{ (i,j) \in D(I) } \left( n(k-i) +\al^i_j \right) &\quad \leq \quad \sum_{j \in J} \nu_j + \sum_{ (d,r) \in D(K)}  n(k-d)
\end{align}
where the last inequalities hold for all LR-consistent triples $I,J,K \in[\ell k]$.
\end{prop}

The above conditions can be simplified to the following. We define an mLR-consistent triple $(I,J,K)$ of subsets of $[1,\ldots,\ell k]$ as an LR-consistent triples satisfying the condition that $|I \cap [ \ell(j-1)+1,\ldots, \ell j]|=|K \cap [\ell(j-1),\ldots, \ell j]|$ for every $j=1,\ldots,k$.

\begin{thm}\label{prop:kron_pos_gen}
    Suppose that $g(\la,\mu,\nu)>0$ and let $\ell=\min\{\ell(\mu),\ell(\nu)\}$. Then there exist nonnegative integers $\{\al^i_j\}_{i\in[k],j\in[\ell]}$ satisfying
     \begin{align}
    \sum_j \al^i_j & =   \la_i, & \text{ for }i\in[k];\\
    \al^i_j & \geq   \al^i_{j+1},  & \text{ for }j\in[\ell-1], \, i\in[k]; \\
    \sum_{ (i,j) \in D(I)} \al^i_j & \leq   \min\{\sum_{j \in J} \mu_j, \sum_{j \in J} \nu_j\}, & \text{ for every mLR-consistent } (I,J,K). 
\end{align}
\end{thm}

\begin{proof}
First note that for $I,J,K$ to be an LR-consistent triple we must have $\rho(K) \subset \rho(I)$, which implies that if $I = \{i_1 < i_2 < \cdots <i_s\}$ and $K=\{ k_1 < \cdots <k_s\}$ then $k_j \leq i_j$ for all $j$. Thus in equation~\eqref{eq:all inequalities} we would have
$$\sum_{ (d,r) \in D(K) } n(k-d) \geq \sum_{(i,j) \in D(I)} n(k-i),$$
with a difference of at least $n$ if the two sums are not actually equal. If they are not equal then the inequalities are trivially satisfied since the sums of $\al^i_j$ and $\mu_i$'s are each within $[0,n]$. We can thus assume that we have an equality. Thus $I= \cup I_p $ and $K= \cup K_p$, where $I_j, K_j \subset [\ell(j-1)+1,\ldots,\ell j]$ and $|I_j|=|K_j|$  for all $j$. For the set $J$ we must have $|J| = |I|$ and $c^{\rho(I)}_{\rho(J)\rho(K)}=1$ , which is the definition of mLR-consistent. \end{proof}

\begin{Cor}\label{cor:kron_pos_2}
    Suppose that $g(\la,\mu,\nu)>0$ and $\ell(\mu)=2$, and set $k=\ell(\la)$. Then there exist nonnegative integers $y_i \in [0, \lfloor \la_i/2 \rfloor ]$ for $i\in[k]$, such that
    \begin{align}
    \sum_{i\in A \cup C} \la_i +\sum_{i \in B }  y_i -\sum_{i\in C} y_i  \leq   \min\{\sum_{j \in J} \mu_j, \sum_{j \in J} \nu_j\}
\end{align}
for all triples of mutually disjoint sets $A \sqcup B \sqcup C \subset[k]$ and $J=\{1,\ldots,r,r+2,\ldots,r+b+1\}$ or $J= \{1,\ldots,r+b-1,r+2b\}$, where $r=2|A|+|C|$ and $b=|B|$.
\end{Cor}

\begin{proof}
    To derive these conditions from Theorem~\ref{prop:kron_pos_gen}, we have $\ell=2$ and we set $x_i:= \al^i_2$, so the partition conditions for $\al^i$ become $0\leq x_i \leq \la_i/2$ as $\al^i_1=\la_i-x_i$. For the sets $I$ and $K$ we must have $I= \cup I_p $ and $K= \cup K_p$, where $I_j, K_j \subset \{2j-1,2j\}$. For each $j$, if $|I_j| = |K_j|$, then either $I_j = K_j$ or $I_j = \{2j\}, K_j =\{2j-1\}$. Then  for such $I,K$, the skew shape $\rho(I)/\rho(K)$ has no two boxes in the same column or row. Then, for any set $J$ with $|J| = |I|$ and $\rho(J)\vdash |\rho(I)/\rho(K)|$, we have $c^{\rho(I)}_{\rho(J)\rho(K)}>0$. This coefficient is equal to 1 only when $\rho(J)=(1,1,\ldots,1)$ or $\rho(J)=(|\rho(I)/\rho(K)|)$ , so $J=\{1,2\ldots,r,r+2,\cdots,r+b+1\}$ or $J = \{1,2\ldots,r+b-1,r+2b\}$ for some $r,b$, respectively.
Now set $A=\{ j: |I_j|=2\}$, set $B=\{j: I_j =\{2j\} \}$ and $C=\{j: I_j = \{2j-1\} \}$, then we have $|J|=|I|=2|A| + |B| + |C|$, and $|\rho(I)/\rho(K)|=|B|=b$, so $r=2|A|+|C|$. Finally, for the indices  in $I$ coming from $j \in A$ we have the terms $\al^j_1 + \al^j_2 = \la_j$, from $B$ we have $\al^j_1 = \la_j-y_j$, and from $B$ we have $y_j$. 
\end{proof}

\begin{rem}
    In particular, taking $A,C=\emptyset$ and $B =\{1\}$ we have $ \lfloor \la_1/2 \rfloor\leq \min(\mu_1,\nu_1)$; taking $A,B=\emptyset$ and $C =\{1\}$ we have $ \la_1- \lfloor \la_1/2 \rfloor\leq \min(\mu_2,\nu_2)$. Summing up the two inequalities recovers $\la_1 \le \min(\mu_1,\nu_1)+\min(\mu_2,\nu_2)$,  equivalent to a result from~\cite{Dvi93} in the given case $\ell(\mu)=2$. 
\end{rem}

\begin{rem}\label{rem:extremal}
    Using this approach we can also understand the maximal partitions $\la$, such that $g(\la,\mu,\nu)>0$ for a fixed pair $\mu,\nu$. Maximal here means with respect to the dominance order, i.e. such that if $\ga \succ \la$ then $g(\ga,\mu,\nu)=0$. First, we have that 
    $$P_k(\mu,\nu):= \{ (\al^1,\al^2,\cdots,\al^k, a_1,\ldots,a_k): (\al^1, \cdots,\al^k) \in P(\mu;a_1,\ldots,a_k) \cap P(\nu;a_1\ldots,a_k)\}$$
    is a convex polytope, described by the set of equalities and inequalities in Corollary~\ref{cor:polytope_horn} for $\mu$ and $\nu$. Projecting $P_k(\mu,\nu)$ to $\mathbb{R}^k$ onto its last $k$ coordinates: $(\al^1,\ldots,\al^k,a_1,\ldots,a_k) \to (a_1,\ldots,a_k)$, results in another polytope $K(\mu,\nu)$, which is exactly $N_k(s_\mu \ast s_\nu)$. Since $M_k(s_\mu \ast s_\nu) = \cup_{\la: g(\la,\mu,\nu)>0} M_k(s_\la)$, and $M_k(s_\la)$ are convex polytopes, we have that the extremal points of that union  must be among the extremal points of the individual polytopes.   So the vertices $(a_1,\cdots, a_k)$ of $K(\mu,\nu)$ for which $a_1 \geq a_2 \geq \cdots \geq a_k$ are among partitions $\la$, for which $g(\la,\mu,\nu)>0$. If $\la$ is not maximal with respect to the dominance order, i.e. there is a $\gamma \succ \la$, such that $g(\ga,\mu,\nu)>0$. Then $N_k(s_\la) \subset N_k(s_\ga)$ and $\la$ is inside, but not an extremal point, of $N_k(s_\ga) \subset K(\mu,\nu)$, and is thus not a vertex of the latter polytope. 
\end{rem}

\section{Additional remarks}\label{s:final}
\subsection{}\label{ss:kron_sat}
The set of partitions $\al$ for which $g(\la,\mu,\al)>0$ is not convex and so is not saturated by our definitions. For example, when $\lambda = (8,8)$ and $\mu = (5,3,1,1,1,1,1,1,1,1)$ we can take $\alpha = (7,3,2,2,2)$, $\beta = (5,5,2,2,2)$ and $\frac{\al+\be}{2}=(6, 4, 2, 2, 2).$ We have that $g(\lambda,\mu, \alpha) =g(\lambda,\mu, \beta)=1$, but $g(\lambda,\mu, \frac{\alpha+\beta}{2}) = 0$. At the same time, $s_\la\ast s_\mu$ does not have a unique maximal element with respect to the dominance order, see Appendix~\ref{appendix:kron_prod}.

This is not surprising, given that the Kronecker coefficients, unlike the Littlewood-Richardson coefficients, fail to satisfy the saturation property. Knutson-Tao's saturation theorem~\cite{KT} gives that $c^\la_{\mu\nu} >0 $ if and only if $c^{N\la}_{N\mu, N\nu}>0$ for every $N\geq 1$. Their proof shows that $c^{\la}_{\mu\nu}>0$ if and only if the associated LR polytope is nonempty because that polytope always has an integer vertex. For the Kronecker coefficients there are no such criteria, and in fact saturation fails already with $g( (1,1),(1,1),(1,1))=0$ but $g( (2,2), (2,2), (2,2))=1$, see further~\cite{BOR09}. Saturation does not hold for the reduced Kronecker coefficients either, see~\cite{PPr} and \cite{IP23}.

\subsection{}\label{ss:complexity}
As shown in~\cite{IMW} the following decision problem \textsf{KronPos} is $\NP$-hard: given input $(\la,\mu,\nu)$ (in unary) determine whether $g(\la,\mu,\nu)>0$. This, in particular, implies that the partition triples with positive Kronecker coefficients cannot be characterized in a computationally efficient manner.  At the same time, the rescaled version is slightly easier: deciding if $g(N\la,N\mu,N\nu)>0$ for some $N$ is in $\NP$ and $\coNP$, as shown in~\cite{burgisser2017membership}. This follows from~\cite{vergne2017inequalities} as such triples can now be characterized by a set of linear inequalities. See also Remark~\ref{rem:extremal} for another possible way of describing extremal points. 

Characterization by determining if a polytope has an integer point does not necessarily make the problem easy (as in computable in polynomial time, i.e. in $\P$). While it is easy to decide if a system of linear inequalities has a real solution, deciding if there is an integral solution, i.e. solving an Integer Linear Program (ILP), in general is an $\NP$-complete problem.

The presence of a monomial $x^\bba$ in $s_\la \ast s_\mu$ appears to be an intermediate problem defined as:

\smallskip
\textsf{KronMonomial:}

Input: $\la,\mu,\bba \in \mathbb{N}_0^k$.

Output: Yes, if and only if $x^\bba$ appears in $s_\la \ast s_\mu (x_1,\ldots,x_k)$. 

\smallskip

{\bf Open Problem.} Is  \textsf{KronMonomial} $\NP$-complete?

\smallskip

Unlike~\textsf{KronPos}, we have that \textsf{KronMonomial} is in $\NP$ and the question is to determine whether it is $\NP$-hard. A witness would be any tuple $(\al^1,\ldots\al^k)$, such that $\al^i \vdash a_i$, together with two multi-LR tableaux of shapes $\la$ and $\mu$ which give $c^\la_{\al^1,\al^2,\ldots}>0$, $c^\mu_{\al^1,\ldots}>0$. Checking if the SSYTs satisfy the ballot conditions can be done in $O(n^2)$ time. 

Answering such computational complexity questions helps us understand the possibilities and limitations for combinatorial interpretations or ``nice'' formulas. For more such problems and formal framework see~\cite{Pak22,Pan23}.

\subsection{} \label{ss:doubts}
We doubt that $s_\la \ast s_\mu$ always has a saturated Newton polytope, but finding counterexamples computationally  becomes quickly infeasible. Studying the polytope $\mathcal{P}(\la,\mu;\bba)=P(\la;\bba)\cap P(\mu;\bba)$ in depth and employing geometric packages like LaTTe would be the next natural step.

\subsection{} \label{ss:plethysm}

In the context of open problems in algebraic combinatorics and representation theory, see, e.g., \cite{plethysm}, the Kronecker and plethysm coefficients often appear in analogous and somewhat unexpected relations. These relations are also evident within Geometric Complexity Theory, see, e.g., \cite{FI, IP17}. It is thus natural to ask whether the plethysm of two Schur functions $s_\la[s_\mu]$ has SNP. Recall that the \defn{plethysm} of two symmetric functions is defined as follows. Suppose that $g=\sum_{\al} x^\al$, where monomials with coefficients larger than 1 just appear multiple times. Then $f[g]:=f(x^{\al^1},x^{\al^2},\cdots)$, substituting all monomials as variables in $f$. In particular
$$s_\la[s_\mu]=\sum_\nu a_\nu(\la[\mu]) s_\nu$$
gives the plethysm coefficients $a_\nu(\la[\mu])$, which are the multiplicites of irreducible $GL$-module $V_\nu$ in the composition 
of the two irreducible representations corresponding to $V_\mu$ and $V_\nu$. Very few cases have been determined, with one of the more general ones being
$$s_{(m)}[s_{2}] = \sum_{\mu \vdash m} s_{2\mu},$$
which clearly has SNP as it contains $s_{(2m)}$. Similarly, $s_{(m)}[s_{(n)}]$ contains the leading term $s_{(mn)}$ and has SNP. It appears that $s_\la[s_\mu]$ always has a unique maximal $s_\nu$ in the expansion, obtained as follows. Let $\mu^1=\mu,\mu^2,\ldots,\mu^\ell$ be the first $\ell=\ell(\la)$ compositions appearing in $N(s_\mu)$ arranged in lexicographic order starting with the maximal. Then there is a monomial of degree  \begin{align}\label{eqn:maximal_nu}
    \nu=\la_1\mu^1 + \la_2 \mu^2+\cdots+\la_\ell \mu^\ell
\end{align} appearing in $s_\la[s_\mu]$. In an earlier version of this paper we conjectured that this is the maximal monomial with respect to the reverse lexicographic order. This was immediately confirmed by Orellana, Saliola, Schilling and Zabrocki in the subsequently released \cite{orellana2024quasi}, see also \cite{iijima2011first,paget2019generalized}. 

In \cite{paget2019generalized}, Paget and Wildon demonstrated that there is a unique maximal term in dominance order in the Schur expansion of $s_{\la}[s_\mu]$ if and only if either $\ell(\la) = 1$ or $\ell(\la) = 2$ and $\mu$ is rectangular. When such a unique maximal term $s_\nu$ with respect to the dominance order is present in the plethysm, it follows that $M(s_\la[s_\mu]) = N(s_\nu)$ and so $s_\la[s_\mu]$ would have a saturated Newton polytope.

The sufficient and necessary condition given by Paget and Wildon also implies that plethysms cannot be expected to have unique maximal terms in general. For instance, $s_{(1,1,1)}[s_{(2,1)}]= s_{(2, 2, 2, 1, 1, 1)} + s_{(3, 2, 2, 1, 1)} + s_{(3, 2, 2, 2)} + s_{(3, 3, 1, 1, 1)} + s_{(3, 3, 2, 1)} + s_{(3, 3, 3)} + s_{(4, 2, 1, 1, 1)} + s_{(4, 2, 2, 1)} + 2*s_{(4, 3, 1, 1)} + s_{(4, 3, 2)} + s_{(4, 4, 1)} + s_{(5, 2, 1, 1)} + s_{(5, 2, 2)} + s_{(5, 3, 1)} + s_{(6, 1, 1, 1)}$. In this plethysm, $(6,1,1,1)$ and $(5,3,1)$ are incomparable maximal elements with respect to the dominance order. Determining whether or not a plethysm with more than one maximal term would have a saturated Newton polytope is far from straightforward. 

{\bf Open problem.} Determine when $s_\la[s_\mu]$ has a saturated Newton polytope.

\subsection*{Acknowledgements}
We are grateful to Christian Ikenmeyer, Allen Knutson, Igor Pak, Anne Schilling and Alexander Yong for helpful discussions and for providing references on the topics. G.P. was partially supported by an NSF CCF:AF grant.

\bibliographystyle{alpha}
\bibliography{bibliography}
\newpage
\appendix

\section{Computations of explicit Kronecker products}\label{appendix:kron_prod}

Let $\lambda = (4,4)$ and $\mu = (2^4)$.
Then $g(2\lambda,2\mu, 2\nu) > 0$ while $g(\lambda,\mu, \nu) = 0$ for \[\nu \in \{(4, 3, 1),  ( 3, 2, 2, 1)\}.\]

Let $\lambda = (6,6)$ and $\mu = (2^6)$.
Then $g(2\lambda,2\mu, 2\nu) > 0$ while $g(\lambda,\mu, \nu) = 0$ for \[\nu \in \{(4, 4, 3, 1),  (3, 3, 3, 1, 1, 1), (4, 3, 2, 2, 1), (3, 2, 2, 2, 2, 1)\}.\]

Let $\lambda = (8,8)$ and $\mu = (2^8)$.
Then $g(2\lambda,2\mu, 2\nu) > 0$ while $g(\lambda,\mu, \nu) = 0$ for \[\nu \in \{(4, 4, 4, 3, 1),  (4, 3, 3, 3, 1, 1, 1), (4, 4, 3, 2, 2, 1), (3, 3, 3, 2, 2, 1, 1, 1), (4, 3, 2, 2, 2, 2, 1), (3, 2, 2, 2, 2, 2, 2, 1)\}.\]

All of the above Kronecker products have SNP by maximal term argument though, as $s_{(4^{k})}$ is the maximal term in dominance order in $s_{(2k,2k)}\ast s_{2^{2k}}$.

Let $\lambda = (8,8)$ and $\mu = (5,3,1,1,1,1,1,1,1,1)$. Let $\alpha = (7,3,2,2,2)$, $\beta = (5,5,2,2,2)$ and $\nu=(6, 4, 2, 2, 2).$ We have that $g(\lambda,\mu, \alpha) =g(\lambda,\mu, \beta)=1$, but $g(\lambda,\mu, \frac{\alpha+\beta}{2}) = 0$.
We have
$s_\lambda \ast s_\mu =  s_{(3, 2, 2, 2, 2, 2, 1, 1, 1)} + s_{(3, 2, 2, 2, 2, 2, 2, 1)} + 2*s_{(3, 3, 2, 2, 2, 1, 1, 1, 1)} + 2*s_{(3, 3, 2, 2, 2, 2, 1, 1)} + 2*s_{(3, 3, 2, 2, 2, 2, 2)} + s_{(3, 3, 3, 2, 1, 1, 1, 1, 1)} + 2*s_{(3, 3, 3, 2, 2, 1, 1, 1)} + 2*s_{(3, 3, 3, 2, 2, 2, 1)} + s_{(3, 3, 3, 3, 2, 1, 1)} + s_{(4, 2, 2, 2, 1, 1, 1, 1, 1, 1)} + 2*s_{(4, 2, 2, 2, 2, 1, 1, 1, 1)} + 4*s_{(4, 2, 2, 2, 2, 2, 1, 1)} + s_{(4, 2, 2, 2, 2, 2, 2)} + s_{(4, 3, 2, 1, 1, 1, 1, 1, 1, 1)} + 4*s_{(4, 3, 2, 2, 1, 1, 1, 1, 1)} + 7*s_{(4, 3, 2, 2, 2, 1, 1, 1)} + 6*s_{(4, 3, 2, 2, 2, 2, 1)} + s_{(4, 3, 3, 1, 1, 1, 1, 1, 1)} + 4*s_{(4, 3, 3, 2, 1, 1, 1, 1)} + 5*s_{(4, 3, 3, 2, 2, 1, 1)} + 3*s_{(4, 3, 3, 2, 2, 2)} + s_{(4, 3, 3, 3, 1, 1, 1)} + s_{(4, 3, 3, 3, 2, 1)} + 2*s_{(4, 4, 2, 1, 1, 1, 1, 1, 1)} + 3*s_{(4, 4, 2, 2, 1, 1, 1, 1)} + 6*s_{(4, 4, 2, 2, 2, 1, 1)} + s_{(4, 4, 2, 2, 2, 2)} + s_{(4, 4, 3, 1, 1, 1, 1, 1)} + 3*s_{(4, 4, 3, 2, 1, 1, 1)} + 3*s_{(4, 4, 3, 2, 2, 1)} + s_{(4, 4, 3, 3, 2)} + s_{(4, 4, 4, 2, 1, 1)} + s_{(5, 2, 2, 1, 1, 1, 1, 1, 1, 1)} + 3*s_{(5, 2, 2, 2, 1, 1, 1, 1, 1)} + 5*s_{(5, 2, 2, 2, 2, 1, 1, 1)} + 4*s_{(5, 2, 2, 2, 2, 2, 1)} + s_{(5, 3, 1, 1, 1, 1, 1, 1, 1, 1)} + 3*s_{(5, 3, 2, 1, 1, 1, 1, 1, 1)} + 8*s_{(5, 3, 2, 2, 1, 1, 1, 1)} + 8*s_{(5, 3, 2, 2, 2, 1, 1)} + 5*s_{(5, 3, 2, 2, 2, 2)} + 2*s_{(5, 3, 3, 1, 1, 1, 1, 1)} + 4*s_{(5, 3, 3, 2, 1, 1, 1)} + 4*s_{(5, 3, 3, 2, 2, 1)} + s_{(5, 3, 3, 3, 1, 1)} + s_{(5, 4, 1, 1, 1, 1, 1, 1, 1)} + 3*s_{(5, 4, 2, 1, 1, 1, 1, 1)} + 5*s_{(5, 4, 2, 2, 1, 1, 1)} + 4*s_{(5, 4, 2, 2, 2, 1)} + s_{(5, 4, 3, 1, 1, 1, 1)} + 2*s_{(5, 4, 3, 2, 1, 1)} + s_{(5, 4, 3, 2, 2)} + s_{(5, 5, 2, 1, 1, 1, 1)} + s_{(5, 5, 2, 2, 2)} + 2*s_{(6, 2, 2, 1, 1, 1, 1, 1, 1)} + 3*s_{(6, 2, 2, 2, 1, 1, 1, 1)} + 5*s_{(6, 2, 2, 2, 2, 1, 1)} + s_{(6, 2, 2, 2, 2, 2)} + s_{(6, 3, 1, 1, 1, 1, 1, 1, 1)} + 3*s_{(6, 3, 2, 1, 1, 1, 1, 1)} + 5*s_{(6, 3, 2, 2, 1, 1, 1)} + 4*s_{(6, 3, 2, 2, 2, 1)} + s_{(6, 3, 3, 1, 1, 1, 1)} + s_{(6, 3, 3, 2, 1, 1)} + s_{(6, 3, 3, 2, 2)} + s_{(6, 4, 1, 1, 1, 1, 1, 1)} + s_{(6, 4, 2, 1, 1, 1, 1)} + 2*s_{(6, 4, 2, 2, 1, 1)} + s_{(7, 2, 2, 1, 1, 1, 1, 1)} + 2*s_{(7, 2, 2, 2, 1, 1, 1)} + 2*s_{(7, 2, 2, 2, 2, 1)} + s_{(7, 3, 2, 1, 1, 1, 1)} + s_{(7, 3, 2, 2, 1, 1)} + s_{(7, 3, 2, 2, 2)} + s_{(8, 2, 2, 2, 1, 1)}.$

\section{The full list of Horn inequalities}\label{appendix_horn}

Here we list the full set of Horn inequalities for $c^\la_{\mu\nu}>0$ when $\ell(\la) =r =6$.

We list the LR-consistent triples with $I, J, K \subset [1,2,3,4,5,6]$ and $c^{\rho(I)}_{\rho(J)\rho(K)}=1$.
\tiny

\begin{tabular}{p{1.8in}p{2in}p{2.5in}}

$I =  [1] , J =  [1] ,K =  [1]  $ 

$I =  [2] , J =  [1] ,K =  [2] $ 

$I =  [3] , J =  [1] ,K =  [3]  $ 

$I =  [4] , J =  [1] ,K =  [4]  $ 

$I =  [5] , J =  [1] ,K =  [5]  $ 

$I =  [6] , J =  [1] ,K =  [6]  $ 

$I =  [3] , J =  [2] ,K =  [2]  $ 

$I =  [4] , J =  [2] ,K =  [3]  $ 

$I =  [5] , J =  [2] ,K =  [4]  $ 

$I =  [6] , J =  [2] ,K =  [5]  $ 

$I =  [5] , J =  [3] ,K =  [3]  $ 

$I =  [6] , J =  [3] ,K =  [4]  $ 

$I =  [2, 1] , J =  [2, 1] ,K =  [2, 1]  $ 

$I =  [3, 1] , J =  [2, 1] ,K =  [3, 1]  $ 

$I =  [4, 1] , J =  [2, 1] ,K =  [4, 1]  $ 

$I =  [5, 1] , J =  [2, 1] ,K =  [5, 1]  $ 

$I =  [6, 1] , J =  [2, 1] ,K =  [6, 1]  $ 

$I =  [3, 2] , J =  [2, 1] ,K =  [3, 2] $ 

$I =  [4, 2] , J =  [2, 1] ,K =  [4, 2]  $ 

$I =  [5, 2] , J =  [2, 1] ,K =  [5, 2]  $ 
&
$I =  [6, 2] , J =  [2, 1] ,K =  [6, 2] $ 

$I =  [4, 3] , J =  [2, 1] ,K =  [4, 3]  $ 

$I =  [5, 3] , J =  [2, 1] ,K =  [5, 3] $ 

$I =  [6, 3] , J =  [2, 1] ,K =  [6, 3]  $ 

$I =  [5, 4] , J =  [2, 1] ,K =  [5, 4]  $ 

$I =  [6, 4] , J =  [2, 1] ,K =  [6, 4]  $ 

$I =  [6, 5] , J =  [2, 1] ,K =  [6, 5]  $ 

$I =  [4, 1] , J =  [3, 1] ,K =  [3, 1]  $ 

$I =  [3, 2] , J =  [3, 1] ,K =  [3, 1]  $ 

$I =  [5, 1] , J =  [3, 1] ,K =  [4, 1]  $ 

$I =  [4, 2] , J =  [3, 1] ,K =  [4, 1]  $ 

$I =  [6, 1] , J =  [3, 1] ,K =  [5, 1]  $ 

$I =  [5, 2] , J =  [3, 1] ,K =  [5, 1]  $ 

$I =  [6, 2] , J =  [3, 1] ,K =  [6, 1]  $ 

$I =  [4, 2] , J =  [3, 1] ,K =  [3, 2]  $ 

$I =  [5, 2] , J =  [3, 1] ,K =  [4, 2]  $ 

$I =  [4, 3] , J =  [3, 1] ,K =  [4, 2]  $ 

$I =  [6, 2] , J =  [3, 1] ,K =  [5, 2]  $ 

$I =  [5, 3] , J =  [3, 1] ,K =  [5, 2]  $ 

$I =  [6, 3] , J =  [3, 1] ,K =  [6, 2]  $ 
&
$I =  [5, 3] , J =  [3, 1] ,K =  [4, 3]  $ 

$I =  [6, 3] , J =  [3, 1] ,K =  [5, 3]  $ 

$I =  [5, 4] , J =  [3, 1] ,K =  [5, 3]  $ 

$I =  [6, 4] , J =  [3, 1] ,K =  [6, 3]  $ 

$I =  [6, 4] , J =  [3, 1] ,K =  [5, 4] $ 

$I =  [6, 5] , J =  [3, 1] ,K =  [6, 4] $ 

$I =  [6, 1] , J =  [4, 1] ,K =  [4, 1]  $ 

$I =  [5, 2] , J =  [4, 1] ,K =  [4, 1]  $ 

$I =  [4, 3] , J =  [4, 1] ,K =  [4, 1]  $ 

$I =  [6, 2] , J =  [4, 1] ,K =  [5, 1]  $ 

$I =  [5, 3] , J =  [4, 1] ,K =  [5, 1] $ 

$I =  [6, 3] , J =  [4, 1] ,K =  [6, 1]  $ 

$I =  [5, 2] , J =  [4, 1] ,K =  [3, 2]  $ 

$I =  [6, 2] , J =  [4, 1] ,K =  [4, 2]  $ 

$I =  [5, 3] , J =  [4, 1] ,K =  [4, 2]  $ 

$I =  [6, 3] , J =  [4, 1] ,K =  [5, 2]  $ 

$I =  [5, 4] , J =  [4, 1] ,K =  [5, 2]  $ 

$I =  [6, 4] , J =  [4, 1] ,K =  [6, 2]  $ 

$I =  [6, 3] , J =  [4, 1] ,K =  [4, 3]  $ 

$I =  [6, 4] , J =  [4, 1] ,K =  [5, 3]  $ 
\end{tabular}

\begin{tabular}{p{1.8in}p{2.1in}p{2.5in}}
$I =  [6, 5] , J =  [4, 1] ,K =  [6, 3]  $ 

$I =  [6, 3] , J =  [5, 1] ,K =  [5, 1]  $ 

$I =  [5, 4] , J =  [5, 1] ,K =  [5, 1]  $ 

$I =  [6, 4] , J =  [5, 1] ,K =  [6, 1]  $ 

$I =  [6, 2] , J =  [5, 1] ,K =  [3, 2]  $ 

$I =  [6, 3] , J =  [5, 1] ,K =  [4, 2]  $ 

$I =  [6, 4] , J =  [5, 1] ,K =  [5, 2]  $ 

$I =  [6, 5] , J =  [5, 1] ,K =  [6, 2]  $ 

$I =  [6, 5] , J =  [6, 1] ,K =  [6, 1]  $ 

$I =  [4, 3] , J =  [3, 2] ,K =  [3, 2]  $ 

$I =  [5, 3] , J =  [3, 2] ,K =  [4, 2]  $ 

$I =  [6, 3] , J =  [3, 2] ,K =  [5, 2]  $ 

$I =  [5, 4] , J =  [3, 2] ,K =  [4, 3]  $ 

$I =  [6, 4] , J =  [3, 2] ,K =  [5, 3]  $ 

$I =  [6, 5] , J =  [3, 2] ,K =  [5, 4]  $ 

$I =  [6, 3] , J =  [4, 2] ,K =  [4, 2]  $ 

$I =  [5, 4] , J =  [4, 2] ,K =  [4, 2] $ 

$I =  [6, 4] , J =  [4, 2] ,K =  [5, 2]  $ 

$I =  [6, 4] , J =  [4, 2] ,K =  [4, 3]  $ 

$I =  [6, 5] , J =  [4, 2] ,K =  [5, 3]  $ 

$I =  [6, 5] , J =  [5, 2] ,K =  [5, 2]  $ 

$I =  [6, 5] , J =  [4, 3] ,K =  [4, 3]  $ 

$I =  [3, 2, 1] , J =  [3, 2, 1] ,K =  [3, 2, 1]  $ 

$I =  [4, 2, 1] , J =  [3, 2, 1] ,K =  [4, 2, 1]  $ 

$I =  [5, 2, 1] , J =  [3, 2, 1] ,K =  [5, 2, 1]  $ 

$I =  [6, 2, 1] , J =  [3, 2, 1] ,K =  [6, 2, 1]  $ 

$I =  [4, 3, 1] , J =  [3, 2, 1] ,K =  [4, 3, 1]  $ 

$I =  [5, 3, 1] , J =  [3, 2, 1] ,K =  [5, 3, 1]  $ 

$I =  [6, 3, 1] , J =  [3, 2, 1] ,K =  [6, 3, 1]  $ 

$I =  [5, 4, 1] , J =  [3, 2, 1] ,K =  [5, 4, 1]  $ 

$I =  [6, 4, 1] , J =  [3, 2, 1] ,K =  [6, 4, 1]  $ 

$I =  [6, 5, 1] , J =  [3, 2, 1] ,K =  [6, 5, 1]  $ 

$I =  [4, 3, 2] , J =  [3, 2, 1] ,K =  [4, 3, 2]  $ 

$I =  [5, 3, 2] , J =  [3, 2, 1] ,K =  [5, 3, 2] $ 

$I =  [6, 3, 2] , J =  [3, 2, 1] ,K =  [6, 3, 2]  $ 

$I =  [5, 4, 2] , J =  [3, 2, 1] ,K =  [5, 4, 2]  $ 

$I =  [6, 4, 2] , J =  [3, 2, 1] ,K =  [6, 4, 2]  $ 

$I =  [6, 5, 2] , J =  [3, 2, 1] ,K =  [6, 5, 2]  $ 

$I =  [5, 4, 3] , J =  [3, 2, 1] ,K =  [5, 4, 3]  $ 

$I =  [6, 4, 3] , J =  [3, 2, 1] ,K =  [6, 4, 3]  $ 

$I =  [6, 5, 3] , J =  [3, 2, 1] ,K =  [6, 5, 3]  $ 

$I =  [6, 5, 4] , J =  [3, 2, 1] ,K =  [6, 5, 4]  $ 

$I =  [5, 2, 1] , J =  [4, 2, 1] ,K =  [4, 2, 1]  $ 

$I =  [4, 3, 1] , J =  [4, 2, 1] ,K =  [4, 2, 1]  $ 

$I =  [6, 2, 1] , J =  [4, 2, 1] ,K =  [5, 2, 1]  $ 

$I =  [5, 3, 1] , J =  [4, 2, 1] ,K =  [5, 2, 1]  $ 

$I =  [6, 3, 1] , J =  [4, 2, 1] ,K =  [6, 2, 1]  $ 

$I =  [5, 3, 1] , J =  [4, 2, 1] ,K =  [4, 3, 1]  $ 

$I =  [4, 3, 2] , J =  [4, 2, 1] ,K =  [4, 3, 1]  $ 

$I =  [6, 3, 1] , J =  [4, 2, 1] ,K =  [5, 3, 1]  $ 

$I =  [5, 4, 1] , J =  [4, 2, 1] ,K =  [5, 3, 1]  $ 

$I =  [5, 3, 2] , J =  [4, 2, 1] ,K =  [5, 3, 1]  $ 

$I =  [6, 4, 1] , J =  [4, 2, 1] ,K =  [6, 3, 1]  $ 

$I =  [6, 3, 2] , J =  [4, 2, 1] ,K =  [6, 3, 1]  $ 

$I =  [6, 4, 1] , J =  [4, 2, 1] ,K =  [5, 4, 1] $ 

$I =  [5, 4, 2] , J =  [4, 2, 1] ,K =  [5, 4, 1]  $ 

$I =  [6, 5, 1] , J =  [4, 2, 1] ,K =  [6, 4, 1]  $ 

$I =  [6, 4, 2] , J =  [4, 2, 1] ,K =  [6, 4, 1]  $ 

$I =  [6, 5, 2] , J =  [4, 2, 1] ,K =  [6, 5, 1]  $ 

$I =  [5, 3, 2] , J =  [4, 2, 1] ,K =  [4, 3, 2]  $ 

$I =  [6, 3, 2] , J =  [4, 2, 1] ,K =  [5, 3, 2]  $ 

$I =  [5, 4, 2] , J =  [4, 2, 1] ,K =  [5, 3, 2]  $ 

$I =  [6, 4, 2] , J =  [4, 2, 1] ,K =  [6, 3, 2]  $ 

$I =  [6, 4, 2] , J =  [4, 2, 1] ,K =  [5, 4, 2]  $ 

$I =  [5, 4, 3] , J =  [4, 2, 1] ,K =  [5, 4, 2]  $ 

$I =  [6, 5, 2] , J =  [4, 2, 1] ,K =  [6, 4, 2] $ 

$I =  [6, 4, 3] , J =  [4, 2, 1] ,K =  [6, 4, 2]  $ 

$I =  [6, 5, 3] , J =  [4, 2, 1] ,K =  [6, 5, 2]  $ 

$I =  [6, 4, 3] , J =  [4, 2, 1] ,K =  [5, 4, 3]  $ 

$I =  [6, 5, 3] , J =  [4, 2, 1] ,K =  [6, 4, 3]  $ 

$I =  [6, 5, 4] , J =  [4, 2, 1] ,K =  [6, 5, 3]  $ 

$I =  [6, 3, 1] , J =  [5, 2, 1] ,K =  [5, 2, 1]  $ 

$I =  [5, 4, 1] , J =  [5, 2, 1] ,K =  [5, 2, 1] $ 

$I =  [6, 4, 1] , J =  [5, 2, 1] ,K =  [6, 2, 1]  $ 

$I =  [6, 3, 1] , J =  [5, 2, 1] ,K =  [4, 3, 1]  $ 
&
$I =  [5, 3, 2] , J =  [5, 2, 1] ,K =  [4, 3, 1]  $ 

$I =  [6, 4, 1] , J =  [5, 2, 1] ,K =  [5, 3, 1]  $ 

$I =  [6, 3, 2] , J =  [5, 2, 1] ,K =  [5, 3, 1]  $ 

$I =  [5, 4, 2] , J =  [5, 2, 1] ,K =  [5, 3, 1]  $ 

$I =  [6, 5, 1] , J =  [5, 2, 1] ,K =  [6, 3, 1]  $ 

$I =  [6, 4, 2] , J =  [5, 2, 1] ,K =  [6, 3, 1]  $ 

$I =  [6, 4, 2] , J =  [5, 2, 1] ,K =  [5, 4, 1]  $ 

$I =  [5, 4, 3] , J =  [5, 2, 1] ,K =  [5, 4, 1] $ 

$I =  [6, 5, 2] , J =  [5, 2, 1] ,K =  [6, 4, 1]  $ 

$I =  [6, 4, 3] , J =  [5, 2, 1] ,K =  [6, 4, 1]  $ 

$I =  [6, 5, 3] , J =  [5, 2, 1] ,K =  [6, 5, 1]  $ 

$I =  [6, 3, 2] , J =  [5, 2, 1] ,K =  [4, 3, 2]  $ 

$I =  [6, 4, 2] , J =  [5, 2, 1] ,K =  [5, 3, 2]  $ 

$I =  [6, 5, 2] , J =  [5, 2, 1] ,K =  [6, 3, 2]  $ 

$I =  [6, 4, 3] , J =  [5, 2, 1] ,K =  [5, 4, 2]  $ 

$I =  [6, 5, 3] , J =  [5, 2, 1] ,K =  [6, 4, 2] $ 

$I =  [6, 5, 4] , J =  [5, 2, 1] ,K =  [6, 5, 2]  $ 

$I =  [6, 5, 1] , J =  [6, 2, 1] ,K =  [6, 2, 1]  $ 

$I =  [6, 3, 2] , J =  [6, 2, 1] ,K =  [4, 3, 1] $ 

$I =  [6, 4, 2] , J =  [6, 2, 1] ,K =  [5, 3, 1]  $ 

$I =  [6, 5, 2] , J =  [6, 2, 1] ,K =  [6, 3, 1]  $ 

$I =  [6, 4, 3] , J =  [6, 2, 1] ,K =  [5, 4, 1]  $ 

$I =  [6, 5, 3] , J =  [6, 2, 1] ,K =  [6, 4, 1] $ 

$I =  [6, 5, 4] , J =  [6, 2, 1] ,K =  [6, 5, 1]  $ 

$I =  [5, 4, 1] , J =  [4, 3, 1] ,K =  [4, 3, 1]  $ 

$I =  [5, 3, 2] , J =  [4, 3, 1] ,K =  [4, 3, 1]  $ 

$I =  [6, 4, 1] , J =  [4, 3, 1] ,K =  [5, 3, 1]  $ 

$I =  [6, 3, 2] , J =  [4, 3, 1] ,K =  [5, 3, 1]  $ 

$I =  [5, 4, 2] , J =  [4, 3, 1] ,K =  [5, 3, 1]  $ 

$I =  [6, 4, 2] , J =  [4, 3, 1] ,K =  [6, 3, 1]  $ 

$I =  [6, 5, 1] , J =  [4, 3, 1] ,K =  [5, 4, 1]  $ 

$I =  [6, 4, 2] , J =  [4, 3, 1] ,K =  [5, 4, 1]  $ 

$I =  [6, 5, 2] , J =  [4, 3, 1] ,K =  [6, 4, 1]  $ 

$I =  [5, 4, 2] , J =  [4, 3, 1] ,K =  [4, 3, 2] , $ 

$I =  [6, 4, 2] , J =  [4, 3, 1] ,K =  [5, 3, 2] , $ 

$I =  [5, 4, 3] , J =  [4, 3, 1] ,K =  [5, 3, 2] , $ 

$I =  [6, 4, 3] , J =  [4, 3, 1] ,K =  [6, 3, 2] ,  $ 

$I =  [6, 5, 2] , J =  [4, 3, 1] ,K =  [5, 4, 2] ,$ 

$I =  [6, 4, 3] , J =  [4, 3, 1] ,K =  [5, 4, 2] , $ 

$I =  [6, 5, 3] , J =  [4, 3, 1] ,K =  [6, 4, 2] ,  $ 

$I =  [6, 5, 3] , J =  [4, 3, 1] ,K =  [5, 4, 3] , $ 

$I =  [6, 5, 4] , J =  [4, 3, 1] ,K =  [6, 4, 3]  $ 

$I =  [6, 5, 1] , J =  [5, 3, 1] ,K =  [5, 3, 1]  $ 

$I =  [5, 4, 3] , J =  [5, 3, 1] ,K =  [5, 3, 1] , $ 

$I =  [6, 5, 2] , J =  [5, 3, 1] ,K =  [6, 3, 1] ,  $ 

$I =  [6, 4, 3] , J =  [5, 3, 1] ,K =  [6, 3, 1] , $ 

$I =  [6, 5, 2] , J =  [5, 3, 1] ,K =  [5, 4, 1] , $ 

$I =  [6, 4, 3] , J =  [5, 3, 1] ,K =  [5, 4, 1] ,  $ 

$I =  [6, 5, 3] , J =  [5, 3, 1] ,K =  [6, 4, 1] ,  $ 

$I =  [6, 4, 2] , J =  [5, 3, 1] ,K =  [4, 3, 2] ,  $ 

$I =  [6, 5, 2] , J =  [5, 3, 1] ,K =  [5, 3, 2] ,  $ 

$I =  [6, 4, 3] , J =  [5, 3, 1] ,K =  [5, 3, 2] , $ 

$I =  [6, 5, 3] , J =  [5, 3, 1] ,K =  [6, 3, 2] ,  $ 

$I =  [6, 5, 3] , J =  [5, 3, 1] ,K =  [5, 4, 2] ,  $ 

$I =  [6, 5, 4] , J =  [5, 3, 1] ,K =  [6, 4, 2] ,  $ 

$I =  [6, 5, 3] , J =  [6, 3, 1] ,K =  [6, 3, 1] , $ 

$I =  [6, 5, 3] , J =  [6, 3, 1] ,K =  [5, 4, 1] , $ 

$I =  [6, 5, 4] , J =  [6, 3, 1] ,K =  [6, 4, 1] ,  $ 

$I =  [6, 5, 2] , J =  [5, 4, 1] ,K =  [4, 3, 2] ,  $ 

$I =  [6, 5, 3] , J =  [5, 4, 1] ,K =  [5, 3, 2] ,  $ 

$I =  [6, 5, 4] , J =  [5, 4, 1] ,K =  [6, 3, 2] ,$ 

$I =  [5, 4, 3] , J =  [4, 3, 2] ,K =  [4, 3, 2] , $ 

$I =  [6, 4, 3] , J =  [4, 3, 2] ,K =  [5, 3, 2] , $ 

$I =  [6, 5, 3] , J =  [4, 3, 2] ,K =  [5, 4, 2] ,  $ 

$I =  [6, 5, 4] , J =  [4, 3, 2] ,K =  [5, 4, 3] , $ 

$I =  [6, 5, 3] , J =  [5, 3, 2] ,K =  [5, 3, 2] ,  $ 

$I =  [6, 5, 4] , J =  [5, 3, 2] ,K =  [5, 4, 2] ,  $ 

$I =  [4, 3, 2, 1] , J =  [4, 3, 2, 1] ,K =  [4, 3, 2, 1]  $ 

$I =  [5, 3, 2, 1] , J =  [4, 3, 2, 1] ,K =  [5, 3, 2, 1]$ 

$I =  [6, 3, 2, 1] , J =  [4, 3, 2, 1] ,K =  [6, 3, 2, 1]  $ 

$I =  [5, 4, 2, 1] , J =  [4, 3, 2, 1] ,K =  [5, 4, 2, 1]  $ 

$I =  [6, 4, 2, 1] , J =  [4, 3, 2, 1] ,K =  [6, 4, 2, 1]  $ 

$I =  [6, 5, 2, 1] , J =  [4, 3, 2, 1] ,K =  [6, 5, 2, 1] $ 

$I =  [5, 4, 3, 1] , J =  [4, 3, 2, 1] ,K =  [5, 4, 3, 1]  $ 

$I =  [6, 4, 3, 1] , J =  [4, 3, 2, 1] ,K =  [6, 4, 3, 1]  $ 

&

$I =  [6, 5, 3, 1] , J =  [4, 3, 2, 1] ,K =  [6, 5, 3, 1] $

$I =  [6, 5, 4, 1] , J =  [4, 3, 2, 1] ,K =  [6, 5, 4, 1] $ 

$I =  [5, 4, 3, 2] , J =  [4, 3, 2, 1] ,K =  [5, 4, 3, 2]  $ 

$I =  [6, 4, 3, 2] , J =  [4, 3, 2, 1] ,K =  [6, 4, 3, 2]$ 

$I =  [6, 5, 3, 2] , J =  [4, 3, 2, 1] ,K =  [6, 5, 3, 2]  $ 

$I =  [6, 5, 4, 2] , J =  [4, 3, 2, 1] ,K =  [6, 5, 4, 2]  $ 

$I =  [6, 5, 4, 3] , J =  [4, 3, 2, 1] ,K =  [6, 5, 4, 3]  $ 

$I =  [6, 3, 2, 1] , J =  [5, 3, 2, 1] ,K =  [5, 3, 2, 1]  $ 

$I =  [5, 4, 2, 1] , J =  [5, 3, 2, 1] ,K =  [5, 3, 2, 1] ,  $ 

$I =  [6, 4, 2, 1] , J =  [5, 3, 2, 1] ,K =  [6, 3, 2, 1] , $ 

$I =  [6, 4, 2, 1] , J =  [5, 3, 2, 1] ,K =  [5, 4, 2, 1] , $ 

$I =  [5, 4, 3, 1] , J =  [5, 3, 2, 1] ,K =  [5, 4, 2, 1] , $ 

$I =  [6, 5, 2, 1] , J =  [5, 3, 2, 1] ,K =  [6, 4, 2, 1] , $ 

$I =  [6, 4, 3, 1] , J =  [5, 3, 2, 1] ,K =  [6, 4, 2, 1] ,  $ 

$I =  [6, 5, 3, 1] , J =  [5, 3, 2, 1] ,K =  [6, 5, 2, 1] ,  $ 

$I =  [6, 4, 3, 1] , J =  [5, 3, 2, 1] ,K =  [5, 4, 3, 1] ,  $ 

$I =  [5, 4, 3, 2] , J =  [5, 3, 2, 1] ,K =  [5, 4, 3, 1] ,  $ 

$I =  [6, 5, 3, 1] , J =  [5, 3, 2, 1] ,K =  [6, 4, 3, 1] ,  $ 

$I =  [6, 4, 3, 2] , J =  [5, 3, 2, 1] ,K =  [6, 4, 3, 1] , $ 

$I =  [6, 5, 4, 1] , J =  [5, 3, 2, 1] ,K =  [6, 5, 3, 1] ,  $ 

$I =  [6, 5, 3, 2] , J =  [5, 3, 2, 1] ,K =  [6, 5, 3, 1] ,  $ 

$I =  [6, 5, 4, 2] , J =  [5, 3, 2, 1] ,K =  [6, 5, 4, 1] ,  $ 

$I =  [6, 4, 3, 2] , J =  [5, 3, 2, 1] ,K =  [5, 4, 3, 2] ,  $ 

$I =  [6, 5, 3, 2] , J =  [5, 3, 2, 1] ,K =  [6, 4, 3, 2] ,  $ 

$I =  [6, 5, 4, 2] , J =  [5, 3, 2, 1] ,K =  [6, 5, 3, 2] ,  $ 

$I =  [6, 5, 4, 3] , J =  [5, 3, 2, 1] ,K =  [6, 5, 4, 2] ,  $ 

$I =  [6, 5, 2, 1] , J =  [6, 3, 2, 1] ,K =  [6, 3, 2, 1] ,  $ 

$I =  [6, 4, 3, 1] , J =  [6, 3, 2, 1] ,K =  [5, 4, 2, 1] ,  $ 

$I =  [6, 5, 3, 1] , J =  [6, 3, 2, 1] ,K =  [6, 4, 2, 1] ,  $ 

$I =  [6, 5, 4, 1] , J =  [6, 3, 2, 1] ,K =  [6, 5, 2, 1] ,  $ 

$I =  [6, 4, 3, 2] , J =  [6, 3, 2, 1] ,K =  [5, 4, 3, 1] ,  $ 

$I =  [6, 5, 3, 2] , J =  [6, 3, 2, 1] ,K =  [6, 4, 3, 1] ,  $ 

$I =  [6, 5, 4, 2] , J =  [6, 3, 2, 1] ,K =  [6, 5, 3, 1] , $ 

$I =  [6, 5, 4, 3] , J =  [6, 3, 2, 1] ,K =  [6, 5, 4, 1] ,  $ 

$I =  [6, 5, 2, 1] , J =  [5, 4, 2, 1] ,K =  [5, 4, 2, 1] ,  $ 

$I =  [6, 4, 3, 1] , J =  [5, 4, 2, 1] ,K =  [5, 4, 2, 1] ,  $ 

$I =  [5, 4, 3, 2] , J =  [5, 4, 2, 1] ,K =  [5, 4, 2, 1] ,$ 

$I =  [6, 5, 3, 1] , J =  [5, 4, 2, 1] ,K =  [6, 4, 2, 1] ,  $ 

$I =  [6, 4, 3, 2] , J =  [5, 4, 2, 1] ,K =  [6, 4, 2, 1] ,  $ 

$I =  [6, 5, 3, 2] , J =  [5, 4, 2, 1] ,K =  [6, 5, 2, 1] ,  $ 

$I =  [6, 5, 3, 1] , J =  [5, 4, 2, 1] ,K =  [5, 4, 3, 1] ,  $ 

$I =  [6, 4, 3, 2] , J =  [5, 4, 2, 1] ,K =  [5, 4, 3, 1] , $ 

$I =  [6, 5, 4, 1] , J =  [5, 4, 2, 1] ,K =  [6, 4, 3, 1] , $ 

$I =  [6, 5, 3, 2] , J =  [5, 4, 2, 1] ,K =  [6, 4, 3, 1] ,  $ 

$I =  [6, 5, 4, 2] , J =  [5, 4, 2, 1] ,K =  [6, 5, 3, 1] , $ 

$I =  [6, 5, 3, 2] , J =  [5, 4, 2, 1] ,K =  [5, 4, 3, 2] ,  $ 

$I =  [6, 5, 4, 2] , J =  [5, 4, 2, 1] ,K =  [6, 4, 3, 2] ,  $ 

$I =  [6, 5, 4, 3] , J =  [5, 4, 2, 1] ,K =  [6, 5, 3, 2] ,  $ 

$I =  [6, 5, 4, 1] , J =  [6, 4, 2, 1] ,K =  [6, 4, 2, 1] ,  $ 

$I =  [6, 5, 3, 2] , J =  [6, 4, 2, 1] ,K =  [6, 4, 2, 1] , $ 

$I =  [6, 5, 4, 2] , J =  [6, 4, 2, 1] ,K =  [6, 5, 2, 1] , $ 

$I =  [6, 5, 3, 2] , J =  [6, 4, 2, 1] ,K =  [5, 4, 3, 1] , $ 

$I =  [6, 5, 4, 2] , J =  [6, 4, 2, 1] ,K =  [6, 4, 3, 1] , $ 

$I =  [6, 5, 4, 3] , J =  [6, 4, 2, 1] ,K =  [6, 5, 3, 1] ,  $ 

$I =  [6, 5, 4, 3] , J =  [6, 5, 2, 1] ,K =  [6, 5, 2, 1] ,  $ 

$I =  [6, 5, 4, 1] , J =  [5, 4, 3, 1] ,K =  [5, 4, 3, 1] ,  $ 

$I =  [6, 5, 3, 2] , J =  [5, 4, 3, 1] ,K =  [5, 4, 3, 1] ,  $ 

$I =  [6, 5, 4, 2] , J =  [5, 4, 3, 1] ,K =  [6, 4, 3, 1] ,  $ 

$I =  [6, 5, 4, 2] , J =  [5, 4, 3, 1] ,K =  [5, 4, 3, 2] ,  $ 

$I =  [6, 5, 4, 3] , J =  [5, 4, 3, 1] ,K =  [6, 4, 3, 2] ,  $ 

$I =  [6, 5, 4, 3] , J =  [6, 4, 3, 1] ,K =  [6, 4, 3, 1] ,  $ 

$I =  [6, 5, 4, 3] , J =  [5, 4, 3, 2] ,K =  [5, 4, 3, 2] ,  $ 

$I =  [5, 4, 3, 2, 1] , J =  [5, 4, 3, 2, 1] ,K =  [5, 4, 3, 2, 1] ,  $ 

$I =  [6, 4, 3, 2, 1] , J =  [5, 4, 3, 2, 1] ,K =  [6, 4, 3, 2, 1] ,  $ 

$I =  [6, 5, 3, 2, 1] , J =  [5, 4, 3, 2, 1] ,K =  [6, 5, 3, 2, 1] , $ 

$I =  [6, 5, 4, 2, 1] , J =  [5, 4, 3, 2, 1] ,K =  [6, 5, 4, 2, 1] , $ 

$I =  [6, 5, 4, 3, 1] , J =  [5, 4, 3, 2, 1] ,K =  [6, 5, 4, 3, 1] , $ 

$I =  [6, 5, 4, 3, 2] , J =  [5, 4, 3, 2, 1] ,K =  [6, 5, 4, 3, 2] ,  $ 

$I =  [6, 5, 3, 2, 1] , J =  [6, 4, 3, 2, 1] ,K =  [6, 4, 3, 2, 1] , $ 

$I =  [6, 5, 4, 2, 1] , J =  [6, 4, 3, 2, 1] ,K =  [6, 5, 3, 2, 1] , $ 

$I =  [6, 5, 4, 3, 1] , J =  [6, 4, 3, 2, 1] ,K =  [6, 5, 4, 2, 1] ,  $ 

$I =  [6, 5, 4, 3, 2] , J =  [6, 4, 3, 2, 1] ,K =  [6, 5, 4, 3, 1] ,  $ 

$I =  [6, 5, 4, 3, 1] , J =  [6, 5, 3, 2, 1] ,K =  [6, 5, 3, 2, 1] , $ 

$I =  [6, 5, 4, 3, 2] , J =  [6, 5, 3, 2, 1] ,K =  [6, 5, 4, 2, 1] ,  $ 
\end{tabular}

\begin{enumerate}
    \item [$s= 1$:] $\lambda_i \le \min_{1\le j \le i}(\mu_j+\nu_{i+1-j})$ for $i \in [6]$
    \item [$s = 2$:] 
    $\lambda_1+\lambda_2 \le \mu_1+\mu_2+\nu_1+\nu_2$

    $\lambda_1+\lambda_3\le \min(\mu_1+\mu_2+\nu_1+\nu_3, \mu_1+\mu_3+\nu_1+\nu_2) $
    
    $\lambda_1+\lambda_4 \le \min(\mu_1+\mu_2+\nu_1+\nu_4, \mu_1+\mu_3+\nu_1+\nu_3, \mu_1+\mu_4+\nu_1+\nu_2) $

    $\lambda_1+\lambda_5 \le \min(\mu_1+\mu_2+\nu_1+\nu_5, \mu_1+\mu_3+\nu_1+\nu_4, \mu_1+\mu_4+\nu_1+\nu_3, \mu_1+\mu_5+\nu_1+\nu_2) $

    $\lambda_1+\lambda_6 \le \min(\mu_1+\mu_2+\nu_1+\nu_6, \mu_1+\mu_3+\nu_1+\nu_5, \mu_1+\mu_4+\nu_1+\nu_4, \mu_1+\mu_5+\nu_1+\nu_3, \mu_1+\mu_6+\nu_1+\nu_2) $

    $\lambda_2+\lambda_3 \le \min(\mu_1+\mu_2+\nu_2+\nu_3, \mu_1+\mu_3+\nu_1+\nu_3,\mu_2+\mu_3+\nu_1+\nu_2) $
    
    $\lambda_2+\lambda_4 \le \min(\mu_1+\mu_2+\nu_2+\nu_4, \mu_1+\mu_3+\nu_1+\nu_4,,\mu_1+\mu_3+\nu_2+\nu_3,\mu_2+\mu_3+\nu_1+\nu_3, \mu_1+\mu_4+\nu_1+\nu_3, \mu_2+\mu_4+\nu_1+\nu_2) $

    $\lambda_2+\lambda_5 \le \min(\mu_1+\mu_2+\nu_2+\nu_5,\mu_1+\mu_3+\nu_1+\nu_5,\mu_1+\mu_3+\nu_2+\nu_4,\mu_2+\mu_4+\nu_1+\nu_3,\mu_1+\mu_4+\nu_2+\nu_3,\mu_1+\mu_4+\nu_1+\nu_4,\mu_2+\mu_3+\nu_1+\nu_4, \mu_1+\mu_5+\nu_1+\nu_3,  , \mu_2+\mu_5+\nu_1+\nu_2) $

    $\lambda_2+\lambda_6 \le \min(\mu_1+\mu_2+\nu_2+\nu_6,\mu_1+\mu_3+\nu_1+\nu_6,\mu_1+\mu_3+\nu_2+\nu_5,\mu_2+\mu_5+\nu_1+\nu_3,\mu_1+\mu_4+\nu_1+\nu_5, \mu_1+\mu_5+\nu_1+\nu_4,\mu_1+\mu_4+\nu_2+\nu_4,\mu_2+\mu_4+\nu_1+\nu_4,\mu_1+\mu_5+\nu_2+\nu_3,\mu_2+\mu_3+\nu_1+\nu_5,\mu_1+\mu_6+\nu_1+\nu_3,  , \mu_2+\mu_6+\nu_1+\nu_2) $

    $\lambda_3+\lambda_4 \le \min(\mu_1+\mu_2+\nu_3+\nu_4, \mu_1+\mu_3+\nu_2+\nu_4,\mu_1+\mu_4+\nu_1+\nu_4,\mu_2+\mu_3+\nu_2+\nu_3, \mu_2+\mu_4+\nu_1+\nu_3,\mu_3+\mu_4+\nu_1+\nu_2) $

    $\lambda_3+\lambda_5 \le \min(\mu_1+\mu_2+\nu_3+\nu_5, \mu_1+\mu_3+\nu_2+\nu_5,\mu_1+\mu_3+\nu_3+\nu_4,\mu_1+\mu_4+\nu_1+\nu_5,\mu_1+\mu_4+\nu_2+\nu_4,\mu_1+\mu_5+\nu_1+\nu_4,\mu_2+\mu_3+\nu_2+\nu_4,\mu_2+\mu_4+\nu_2+\nu_3,\mu_2+\mu_4+\nu_1+\nu_4,\mu_3+\mu_4+\nu_1+\nu_3, \mu_2+\mu_5+\nu_1+\nu_3,\mu_3+\mu_5+\nu_1+\nu_2) $

    $\lambda_3+\lambda_6 \le \min(
    \mu_1+\mu_2+\nu_3+\nu_6, 
    \mu_1+\mu_3+\nu_2+\nu_6,
    \mu_1+\mu_3+\nu_3+\nu_5,
    \mu_1+\mu_4+\nu_1+\nu_6,
    \mu_1+\mu_4+\nu_2+\nu_5,
    \mu_1+\mu_4+\nu_3+\nu_4,
    \mu_1+\mu_5+\nu_1+\nu_5,
    \mu_1+\mu_5+\nu_2+\nu_4,
    \mu_1+\mu_6+\nu_1+\nu_4,
    \mu_2+\mu_3+\nu_2+\nu_5,
    \mu_2+\mu_4+\nu_1+\nu_5,
    \mu_2+\mu_4+\nu_2+\nu_4,
    \mu_2+\mu_5+\nu_2+\nu_3,
    \mu_2+\mu_5+\nu_1+\nu_4,
    \mu_2+\mu_6+\nu_1+\nu_3,
    \mu_3+\mu_4+\nu_1+\nu_4, 
    \mu_3+\mu_5+\nu_1+\nu_3, 
    \mu_3+\mu_6+\nu_1+\nu_2) $
    
     $\lambda_4+\lambda_5 \le \min(
    \mu_1+\mu_2+\nu_4+\nu_5, 
    \mu_1+\mu_3+\nu_3+\nu_5, 
    \mu_1+\mu_4+\nu_2+\nu_5, 
    \mu_1+\mu_5+\nu_1+\nu_5, 
    \mu_2+\mu_3+\nu_3+\nu_4, 
    \mu_2+\mu_4+\nu_2+\nu_4, 
    \mu_2+\mu_5+\nu_1+\nu_4, 
    \mu_3+\mu_4+\nu_2+\nu_3, 
    \mu_3+\mu_5+\nu_1+\nu_3, 
    \mu_4+\mu_5+\nu_1+\nu_2) $
    
    $\lambda_4+\lambda_6 \le \min(
    \mu_1+\mu_2+\nu_4+\nu_6, 
    \mu_1+\mu_3+\nu_3+\nu_6,
    \mu_1+\mu_3+\nu_4+\nu_5,
    \mu_1+\mu_4+\nu_2+\nu_6,
    \mu_1+\mu_5+\nu_1+\nu_6,
    \mu_1+\mu_5+\nu_2+\nu_5,
    \mu_1+\mu_6+\nu_1+\nu_5,
    \mu_2+\mu_3+\nu_3+\nu_5,
    \mu_2+\mu_4+\nu_3+\nu_4,
    \mu_2+\mu_4+\nu_2+\nu_5,
    \mu_2+\mu_5+\nu_2+\nu_4,
    \mu_2+\mu_5+\nu_1+\nu_5,
    \mu_2+\mu_6+\nu_1+\nu_4,
    \mu_3+\mu_4+\nu_2+\nu_4,
    \mu_3+\mu_5+\nu_2+\nu_3,
    \mu_3+\mu_5+\nu_1+\nu_4,
    \mu_3+\mu_6+\nu_1+\nu_3, 
    \mu_4+\mu_5+\nu_1+\nu_3, 
    \mu_4+\mu_6+\nu_1+\nu_2) $

     $\lambda_5+\lambda_6 \le \min(
    \mu_1+\mu_2+\nu_5+\nu_6, 
    \mu_1+\mu_3+\nu_4+\nu_6,
    \mu_1+\mu_4+\nu_3+\nu_6,
    \mu_1+\mu_5+\nu_2+\nu_6,
    \mu_1+\mu_6+\nu_1+\nu_6,
    \mu_2+\mu_6+\nu_1+\nu_5,
    \mu_2+\mu_3+\nu_4+\nu_5,
    \mu_2+\mu_4+\nu_3+\nu_5,
    \mu_2+\mu_5+\nu_2+\nu_5,
    \mu_3+\mu_4+\nu_3+\nu_4,
    \mu_3+\mu_5+\nu_2+\nu_4,
    \mu_4+\mu_5+\nu_2+\nu_3,
    \mu_3+\mu_6+\nu_1+\nu_4, 
    \mu_4+\mu_6+\nu_1+\nu_3, 
    \mu_5+\mu_6+\nu_1+\nu_2) $

    \item [$s=3$:] 
    $\lambda_1+\lambda_2+\lambda_i \le \min(\mu_1+\mu_2+\mu_3+
    \nu_1+\nu_2+\nu_i,\mu_1+\mu_2+\mu_i+
    \nu_1+\nu_2+\nu_3)$ for $i \in \{3,4\}$
    $\lambda_1+\lambda_2+\lambda_5 \le \min(\mu_1+\mu_2+\mu_3+\nu_1+\nu_2+\nu_5, \mu_1+\mu_2+\mu_4+\nu_1+\nu_2+\nu_4, \mu_1+\mu_2+\mu_5+\nu_1+\nu_2+\nu_3)$

    $\lambda_1+\lambda_2+\lambda_6 \le \min(\mu_1+\mu_2+\mu_3+\nu_1+\nu_2+\nu_6, \mu_1+\mu_2+\mu_4+\nu_1+\nu_2+\nu_5, 
    \mu_1+\mu_2+\mu_5+\nu_1+\nu_2+\nu_4, \mu_1+\mu_2+\mu_6+\nu_1+\nu_2+\nu_3)$

    $\lambda_1+\lambda_3+\lambda_4 \le \min(\mu_1+\mu_2+\mu_3+\nu_1+\nu_3+\nu_4, \mu_1+\mu_2+\mu_4+\nu_1+\nu_2+\nu_4, 
    \mu_1+\mu_3+\mu_4+\nu_1+\nu_2+\nu_3)$

    $\lambda_1+\lambda_3+\lambda_5 \le \min(\mu_1+\mu_2+\mu_3+\nu_1+\nu_3+\nu_5, \mu_1+\mu_2+\mu_4+\nu_1+\nu_2+\nu_5, 
    \mu_1+\mu_2+\mu_4+\nu_1+\nu_3+\nu_4, 
    \mu_1+\mu_2+\mu_5+\nu_1+\nu_2+\nu_4,
    \mu_1+\mu_3+\mu_4+\nu_1+\nu_2+\nu_4, 
    \mu_1+\mu_3+\mu_5+\nu_1+\nu_2+\nu_3)$

    $\lambda_1+\lambda_3+\lambda_6 \le \min(\mu_1+\mu_2+\mu_3+\nu_1+\nu_3+\nu_6, \mu_1+\mu_2+\mu_4+\nu_1+\nu_2+\nu_6, 
    \mu_1+\mu_2+\mu_4+\nu_1+\nu_3+\nu_5, 
    \mu_1+\mu_2+\mu_5+\nu_1+\nu_2+\nu_5,
    \mu_1+\mu_2+\mu_5+\nu_1+\nu_3+\nu_4,
    \mu_1+\mu_2+\mu_6+\nu_1+\nu_2+\nu_4,
    \mu_1+\mu_3+\mu_4+\nu_1+\nu_2+\nu_5, 
    \mu_1+\mu_3+\mu_5+\nu_1+\nu_2+\nu_4, 
    \mu_1+\mu_3+\mu_6+\nu_1+\nu_2+\nu_3)$
    
    $\lambda_1+\lambda_4+\lambda_5 \le \min(\mu_1+\mu_2+\mu_3+\nu_1+\nu_4+\nu_5, \mu_1+\mu_2+\mu_4+\nu_1+\nu_3+\nu_5, 
    \mu_1+\mu_2+\mu_5+\nu_1+\nu_2+\nu_5,
    \mu_1+\mu_3+\mu_4+\nu_1+\nu_3+\nu_4,
    \mu_1+\mu_3+\mu_5+\nu_1+\nu_2+\nu_4, 
    \mu_1+\mu_4+\mu_5+\nu_1+\nu_2+\nu_3)$

    $\lambda_1+\lambda_4+\lambda_6 \le \min(\mu_1+\mu_2+\mu_3+\nu_1+\nu_4+\nu_6, 
    \mu_1+\mu_2+\mu_4+\nu_1+\nu_3+\nu_6,    \mu_1+\mu_2+\mu_4+\nu_1+\nu_4+\nu_5,
     \mu_1+\mu_2+\mu_5+\nu_1+\nu_2+\nu_6,  
    \mu_1+\mu_2+\mu_5+\nu_1+\nu_3+\nu_5,  
    \mu_1+\mu_3+\mu_4+\nu_1+\nu_3+\nu_5,  
    \mu_1+\mu_3+\mu_5+\nu_1+\nu_3+\nu_4,  
    \mu_1+\mu_3+\mu_5+\nu_1+\nu_2+\nu_5,  
     \mu_1+\mu_2+\mu_6+\nu_1+\nu_2+\nu_5,
    \mu_1+\mu_4+\mu_5+\nu_1+\nu_2+\nu_4, 
    \mu_1+\mu_3+\mu_6+\nu_1+\nu_2+\nu_4, 
    \mu_1+\mu_4+\mu_6+\nu_1+\nu_2+\nu_3)$
    
    $\lambda_1+\lambda_5+\lambda_6 \le \min(\mu_1+\mu_2+\mu_3+\nu_1+\nu_5+\nu_6,
    \mu_1+\mu_2+\mu_4+\nu_1+\nu_4+\nu_6,
    \mu_1+\mu_2+\mu_5+\nu_1+\nu_3+\nu_6,  
    \mu_1+\mu_2+\mu_6+\nu_1+\nu_2+\nu_6,
        \mu_1+\mu_3+\mu_4+\nu_1+\nu_4+\nu_5,  
    \mu_1+\mu_3+\mu_5+\nu_1+\nu_3+\nu_5,  
    \mu_1+\mu_4+\mu_5+\nu_1+\nu_3+\nu_4, 
    \mu_1+\mu_3+\mu_6+\nu_1+\nu_2+\nu_5,  
    \mu_1+\mu_4+\mu_6+\nu_1+\nu_2+\nu_4, 
    \mu_1+\mu_5+\mu_6+\nu_1+\nu_2+\nu_3)$

    $\lambda_2+\lambda_3+\lambda_4 \le \min(\mu_1+\mu_2+\mu_3+\nu_2+\nu_3+\nu_4,
    \mu_1+\mu_2+\mu_4+\nu_1+\nu_3+\nu_4,
    \mu_1+\mu_3+\mu_4+\nu_1+\nu_2+\nu_4,    
    \mu_2+\mu_3+\mu_4+\nu_1+\nu_2+\nu_3)$
    
    $\lambda_2+\lambda_3+\lambda_5 \le \min(\mu_1+\mu_2+\mu_3+\nu_2+\nu_3+\nu_5,
    \mu_1+\mu_2+\mu_4+\nu_1+\nu_3+\nu_5,
    \mu_1+\mu_2+\mu_4+\nu_2+\nu_3+\nu_4,
    \mu_1+\mu_2+\mu_5+\nu_1+\nu_3+\nu_4,
    \mu_1+\mu_3+\mu_4+\nu_1+\nu_2+\nu_5,   
        \mu_1+\mu_3+\mu_4+\nu_1+\nu_3+\nu_4,   
    \mu_2+\mu_3+\mu_4+\nu_1+\nu_2+\nu_4,    
    \mu_1+\mu_3+\mu_5+\nu_1+\nu_2+\nu_4,    
    \mu_2+\mu_3+\mu_5+\nu_1+\nu_2+\nu_3)$
    
    $\lambda_2+\lambda_3+\lambda_6 \le \min(\mu_1+\mu_2+\mu_3+\nu_2+\nu_3+\nu_6,
    \mu_1+\mu_2+\mu_4+\nu_1+\nu_3+\nu_6,
    \mu_1+\mu_2+\mu_4+\nu_2+\nu_3+\nu_5,
    \mu_1+\mu_2+\mu_4+\nu_2+\nu_3+\nu_5,
        \mu_1+\mu_2+\mu_5+\nu_1+\nu_3+\nu_5,
    \mu_1+\mu_2+\mu_6+\nu_1+\nu_3+\nu_4,
        \mu_1+\mu_3+\mu_5+\nu_1+\nu_3+\nu_4,
        \mu_1+\mu_3+\mu_4+\nu_1+\nu_3+\nu_5,
        \mu_1+\mu_2+\mu_5+\nu_2+\nu_3+\nu_4,
        \mu_2+\mu_3+\mu_4+\nu_1+\nu_2+\nu_5,
    \mu_1+\mu_3+\mu_5+\nu_1+\nu_2+\nu_5,
        \mu_1+\mu_3+\mu_4+\nu_1+\nu_2+\nu_6,
        \mu_2+\mu_3+\mu_5+\nu_1+\nu_2+\nu_4,    
    \mu_2+\mu_3+\mu_5+\nu_1+\nu_2+\nu_4,   
    \mu_1+\mu_3+\mu_6+\nu_1+\nu_2+\nu_4,    
    \mu_2+\mu_3+\mu_6+\nu_1+\nu_2+\nu_3)$
       
    $\lambda_2+\lambda_4+\lambda_5 \le \min(\mu_1+\mu_2+\mu_3+\nu_2+\nu_4+\nu_5,
    \mu_1+\mu_2+\mu_4+\nu_1+\nu_4+\nu_5,
    \mu_1+\mu_2+\mu_4+\nu_2+\nu_3+\nu_5,
    \mu_1+\mu_2+\mu_5+\nu_1+\nu_3+\nu_5,
        \mu_1+\mu_3+\mu_4+\nu_1+\nu_3+\nu_5,
    \mu_2+\mu_3+\mu_4+\nu_1+\nu_3+\nu_4,
    \mu_1+\mu_3+\mu_4+\nu_2+\nu_3+\nu_4,
        \mu_1+\mu_3+\mu_5+\nu_1+\nu_3+\nu_4,
        \mu_1+\mu_3+\mu_5+\nu_1+\nu_2+\nu_5,    
    \mu_2+\mu_3+\mu_5+\nu_1+\nu_2+\nu_4,   
    \mu_1+\mu_4+\mu_5+\nu_1+\nu_2+\nu_4,
    \mu_2+\mu_4+\mu_5+\nu_1+\nu_2+\nu_3)$

    $\lambda_2+\lambda_4+\lambda_6 \le \min(\mu_1+\mu_2+\mu_3+\nu_2+\nu_4+\nu_6,
    \mu_1+\mu_2+\mu_4+\nu_1+\nu_4+\nu_6,
    \mu_1+\mu_2+\mu_4+\nu_2+\nu_3+\nu_6,
        \mu_1+\mu_2+\mu_4+\nu_2+\nu_4+\nu_5,
    \mu_1+\mu_2+\mu_5+\nu_1+\nu_3+\nu_6,
    \mu_1+\mu_2+\mu_5+\nu_1+\nu_4+\nu_5,
    \mu_1+\mu_2+\mu_5+\nu_2+\nu_3+\nu_5,    \mu_1+\mu_3+\mu_6+\nu_1+\nu_2+\nu_5,
    \mu_1+\mu_4+\mu_5+\nu_1+\nu_2+\nu_5,
    \mu_2+\mu_3+\mu_5+\nu_1+\nu_2+\nu_5,
     \mu_1+\mu_3+\mu_5+\nu_2+\nu_3+\nu_4,
    \mu_2+\mu_3+\mu_4+\nu_1+\nu_3+\nu_5,
    \mu_1+\mu_3+\mu_5+\nu_1+\nu_3+\nu_5,    
    \mu_1+\mu_2+\mu_6+\nu_1+\nu_3+\nu_5,
    \mu_1+\mu_3+\mu_5+\nu_1+\nu_2+\nu_6,
    \mu_1+\mu_3+\mu_4+\nu_2+\nu_3+\nu_6,    \mu_1+\mu_3+\mu_6+\nu_1+\nu_3+\nu_4,
    \mu_1+\mu_4+\mu_5+\nu_1+\nu_3+\nu_4,
    \mu_2+\mu_3+\mu_5+\nu_1+\nu_3+\nu_4,
    \mu_1+\mu_3+\mu_4+\nu_1+\nu_4+\nu_5,
    \mu_1+\mu_3+\mu_4+\nu_2+\nu_3+\nu_5,   
    \mu_2+\mu_3+\mu_6+\nu_1+\nu_2+\nu_4,
    \mu_2+\mu_4+\mu_5+\nu_1+\nu_2+\nu_4,
    \mu_1+\mu_4+\mu_6+\nu_1+\nu_2+\nu_4,
    \mu_2+\mu_4+\mu_6+\nu_1+\nu_2+\nu_3)$

    $\lambda_2+\lambda_5+\lambda_6 \le \min(\mu_1+\mu_2+\mu_3+\nu_2+\nu_5+\nu_6,
    \mu_1+\mu_2+\mu_4+\nu_1+\nu_5+\nu_6,
    \mu_1+\mu_2+\mu_4+\nu_2+\nu_4+\nu_6,
    \mu_1+\mu_2+\mu_5+\nu_1+\nu_4+\nu_6,
    \mu_1+\mu_2+\mu_5+\nu_2+\nu_3+\nu_6, 
    \mu_1+\mu_2+\mu_6+\nu_1+\nu_3+\nu_6,
    \mu_1+\mu_3+\mu_4+\nu_1+\nu_4+\nu_6,  
    \mu_1+\mu_3+\mu_4+\nu_2+\nu_4+\nu_5,
    \mu_1+\mu_3+\mu_5+\nu_1+\nu_3+\nu_6,
    \mu_1+\mu_3+\mu_6+\nu_1+\nu_3+\nu_5,
    \mu_1+\mu_3+\mu_5+\nu_1+\nu_4+\nu_5,
    \mu_1+\mu_4+\mu_5+\nu_1+\nu_3+\nu_5,
    \mu_1+\mu_3+\mu_5+\nu_2+\nu_3+\nu_5,
    \mu_2+\mu_3+\mu_5+\nu_1+\nu_3+\nu_5,
    \mu_1+\mu_4+\mu_5+\nu_2+\nu_3+\nu_4,
    \mu_2+\mu_3+\mu_4+\nu_1+\nu_4+\nu_5,
    \mu_2+\mu_4+\mu_5+\nu_1+\nu_3+\nu_4,
    \mu_1+\mu_4+\mu_6+\nu_1+\nu_3+\nu_4, 
    \mu_1+\mu_3+\mu_6+\nu_1+\nu_2+\nu_6,
    \mu_2+\mu_3+\mu_6+\nu_1+\nu_2+\nu_5,
    \mu_1+\mu_4+\mu_6+\nu_1+\nu_2+\nu_5,
    \mu_2+\mu_4+\mu_6+\nu_1+\nu_2+\nu_4,
    \mu_1+\mu_5+\mu_6+\nu_1+\nu_2+\nu_4,
    \mu_2+\mu_5+\mu_6+\nu_1+\nu_2+\nu_3)$
    
    $\lambda_3+\lambda_4+\lambda_5 \le \min(\mu_1+\mu_2+\mu_3+\nu_3+\nu_4+\nu_5,
    \mu_1+\mu_2+\mu_4+\nu_2+\nu_4+\nu_5,
    \mu_1+\mu_2+\mu_5+\nu_1+\nu_4+\nu_5,
    \mu_1+\mu_3+\mu_4+\nu_2+\nu_3+\nu_5,
    \mu_1+\mu_3+\mu_5+\nu_1+\nu_3+\nu_5,
    \mu_2+\mu_3+\mu_4+\nu_2+\nu_3+\nu_4,
    \mu_2+\mu_3+\mu_5+\nu_1+\nu_3+\nu_4,
    \mu_1+\mu_4+\mu_5+\nu_1+\nu_2+\nu_5,
    \mu_2+\mu_4+\mu_5+\nu_1+\nu_2+\nu_4,
    \mu_3+\mu_4+\mu_5+\nu_1+\nu_2+\nu_3)$
    
    $\lambda_3+\lambda_4+\lambda_6 \le \min(\mu_1+\mu_2+\mu_3+\nu_3+\nu_4+\nu_6,
    \mu_1+\mu_2+\mu_4+\nu_2+\nu_4+\nu_6,
    \mu_1+\mu_2+\mu_4+\nu_3+\nu_4+\nu_5,
    \mu_1+\mu_2+\mu_5+\nu_1+\nu_4+\nu_6,
    \mu_1+\mu_2+\mu_5+\nu_2+\nu_4+\nu_5,
    \mu_1+\mu_2+\mu_6+\nu_1+\nu_4+\nu_5,
    \mu_1+\mu_3+\mu_4+\nu_2+\nu_3+\nu_6,
    \mu_1+\mu_3+\mu_5+\nu_1+\nu_3+\nu_6,
    \mu_1+\mu_3+\mu_6+\nu_1+\nu_3+\nu_5,
    \mu_1+\mu_3+\mu_5+\nu_1+\nu_4+\nu_5,
    \mu_1+\mu_4+\mu_5+\nu_1+\nu_3+\nu_5,
    \mu_2+\mu_3+\mu_4+\nu_2+\nu_3+\nu_5,
    \mu_2+\mu_3+\mu_5+\nu_2+\nu_3+\nu_4,
    \mu_1+\mu_3+\mu_5+\nu_2+\nu_3+\nu_5,
    \mu_2+\mu_3+\mu_5+\nu_1+\nu_3+\nu_5,
    \mu_1+\mu_3+\mu_4+\nu_2+\nu_4+\nu_5,
    \mu_2+\mu_4+\mu_5+\nu_1+\nu_3+\nu_4,
    \mu_2+\mu_3+\mu_6+\nu_1+\nu_3+\nu_4,
    \mu_1+\mu_4+\mu_5+\nu_1+\nu_2+\nu_6,
    \mu_2+\mu_4+\mu_5+\nu_1+\nu_2+\nu_5,
    \mu_1+\mu_4+\mu_6+\nu_1+\nu_2+\nu_5,
    \mu_3+\mu_4+\mu_5+\nu_1+\nu_2+\nu_4,
    \mu_2+\mu_4+\mu_6+\nu_1+\nu_2+\nu_4,
    \mu_3+\mu_4+\mu_6+\nu_1+\nu_2+\nu_3)$    
    
    $\lambda_3+\lambda_5+\lambda_6 \le \min(\mu_1+\mu_2+\mu_3+\nu_3+\nu_5+\nu_6,
    \mu_1+\mu_2+\mu_4+\nu_2+\nu_5+\nu_6,
    \mu_1+\mu_2+\mu_4+\nu_3+\nu_4+\nu_6,
    \mu_1+\mu_2+\mu_5+\nu_1+\nu_5+\nu_6,
    \mu_1+\mu_2+\mu_5+\nu_2+\nu_4+\nu_6,
    \mu_1+\mu_2+\mu_6+\nu_1+\nu_4+\nu_6,
    \mu_1+\mu_3+\mu_4+\nu_2+\nu_4+\nu_6,
    \mu_1+\mu_3+\mu_4+\nu_3+\nu_4+\nu_5,
    \mu_1+\mu_4+\mu_6+\nu_1+\nu_3+\nu_5,
    \mu_1+\mu_3+\mu_5+\nu_1+\nu_4+\nu_6,
    \mu_1+\mu_3+\mu_5+\nu_2+\nu_3+\nu_6,
    \mu_2+\mu_3+\mu_6+\nu_1+\nu_3+\nu_5,
    \mu_1+\mu_3+\mu_5+\nu_2+\nu_4+\nu_5,
    \mu_2+\mu_4+\mu_5+\nu_1+\nu_3+\nu_5,
    \mu_1+\mu_3+\mu_6+\nu_1+\nu_3+\nu_6,
    \mu_1+\mu_3+\mu_6+\nu_2+\nu_4+\nu_5,
    \mu_1+\mu_4+\mu_5+\nu_1+\nu_3+\nu_6,
    \mu_1+\mu_4+\mu_5+\nu_2+\nu_3+\nu_5,
    \mu_2+\mu_3+\mu_5+\nu_1+\nu_4+\nu_5,
    \mu_2+\mu_3+\mu_4+\nu_2+\nu_4+\nu_5,
    \mu_2+\mu_4+\mu_5+\nu_2+\nu_3+\nu_4,
    \mu_2+\mu_3+\mu_5+\nu_2+\nu_3+\nu_5,
    \mu_3+\mu_4+\mu_5+\nu_1+\nu_3+\nu_4,
    \mu_2+\mu_4+\mu_6+\nu_1+\nu_3+\nu_4,
    \mu_1+\mu_4+\mu_6+\nu_1+\nu_2+\nu_6,
    \mu_2+\mu_4+\mu_6+\nu_1+\nu_2+\nu_5,    
    \mu_1+\mu_5+\mu_6+\nu_1+\nu_2+\nu_5,
    \mu_3+\mu_4+\mu_6+\nu_1+\nu_2+\nu_4,
    \mu_2+\mu_5+\mu_6+\nu_1+\nu_2+\nu_4,
    \mu_3+\mu_5+\mu_6+\nu_1+\nu_2+\nu_3)$

    $\lambda_4+\lambda_5+\lambda_6 \le \min(\mu_1+\mu_2+\mu_3+\nu_4+\nu_5+\nu_6,   
    \mu_1+\mu_2+\mu_4+\nu_3+\nu_5+\nu_6,
    \mu_1+\mu_2+\mu_5+\nu_2+\nu_5+\nu_6,
    \mu_1+\mu_2+\mu_6+\nu_1+\nu_5+\nu_6,
    \mu_1+\mu_3+\mu_4+\nu_3+\nu_4+\nu_6,
    \mu_1+\mu_3+\mu_5+\nu_2+\nu_4+\nu_6,
    \mu_1+\mu_3+\mu_4+\nu_2+\nu_4+\nu_6,
    \mu_1+\mu_4+\mu_6+\nu_1+\nu_3+\nu_6,
    \mu_1+\mu_3+\mu_6+\nu_1+\nu_4+\nu_6,
    \mu_1+\mu_4+\mu_5+\nu_2+\nu_3+\nu_6,
    \mu_2+\mu_3+\mu_6+\nu_1+\nu_4+\nu_5,
    \mu_2+\mu_3+\mu_4+\nu_3+\nu_4+\nu_5,
    \mu_3+\mu_4+\mu_5+\nu_2+\nu_3+\nu_4,
    \mu_2+\mu_3+\mu_5+\nu_2+\nu_4+\nu_5,
    \mu_2+\mu_4+\mu_5+\nu_2+\nu_3+\nu_5,
    \mu_2+\mu_4+\mu_6+\nu_1+\nu_3+\nu_5,
    \mu_3+\mu_4+\mu_6+\nu_1+\nu_3+\nu_4,
    \mu_1+\mu_5+\mu_6+\nu_1+\nu_2+\nu_6,
    \mu_2+\mu_5+\mu_6+\nu_1+\nu_2+\nu_5,
    \mu_3+\mu_5+\mu_6+\nu_1+\nu_2+\nu_4,    
    \mu_4+\mu_5+\mu_6+\nu_1+\nu_2+\nu_3)$

    \item [$s=4$:] 
    $\lambda_5+\lambda_6 \ge \mu_5+\mu_6+\nu_5+\nu_6$
    
    $\lambda_4+\lambda_6 \ge \max(\mu_4+\mu_6+\nu_5+\nu_6,\mu_5+\mu_6+\nu_4+\nu_6)$

    $\lambda_4+\lambda_5 \ge \max(\mu_4+\mu_5+\nu_5+\nu_6,\mu_4+\mu_6+\nu_4+\nu_6, \mu_5+\mu_6+\nu_4+\nu_5)$

    $\lambda_3+\lambda_6 \ge \max(\mu_3+\mu_6+\nu_5+\nu_6,\mu_4+\mu_6+\nu_4+\nu_6, \mu_5+\mu_6+\nu_3+\nu_6)$

    $\lambda_3+\lambda_5 \ge \max(\mu_3+\mu_5+\nu_5+\nu_6,\mu_4+\mu_5+\nu_4+\nu_6,\mu_4+\mu_6+\nu_4+\nu_5, \mu_4+\mu_6+\nu_3+\nu_6,\mu_3+\mu_6+\nu_4+\nu_6,\mu_5+\mu_6+\nu_3+\nu_5)$

    $\lambda_3+\lambda_4 \ge \max(\mu_3+\mu_4+\nu_5+\nu_6,
    \mu_3+\mu_5+\nu_4+\nu_6,
    \mu_4+\mu_6+\nu_3+\nu_5, 
    \mu_4+\mu_5+\nu_4+\nu_5, 
    \mu_3+\mu_6+\nu_3+\nu_6,
    \mu_5+\mu_6+\nu_3+\nu_4)$
    
    $\lambda_2+\lambda_6 \ge 
    \max(\mu_2+\mu_6+\nu_5+\nu_6,
    \mu_3+\mu_6+\nu_4+\nu_6,
    \mu_4+\mu_6+\nu_3+\nu_6, 
    \mu_5+\mu_6+\nu_2+\nu_6)$
    
    $\lambda_2+\lambda_5 \ge 
    \max(\mu_2+\mu_5+\nu_5+\nu_6,
    \mu_3+\mu_5+\nu_4+\nu_6,
    \mu_4+\mu_6+\nu_3+\nu_5, 
    \mu_2+\mu_6+\nu_4+\nu_6,
    \mu_4+\mu_6+\nu_2+\nu_6,     
    \mu_3+\mu_6+\nu_4+\nu_5,
    \mu_3+\mu_6+\nu_3+\nu_6,
    \mu_4+\mu_5+\nu_3+\nu_6, 
    \mu_5+\mu_6+\nu_2+\nu_5)$

    $\lambda_2+\lambda_4 \ge 
    \max(\mu_2+\mu_4+\nu_5+\nu_6,
    \mu_2+\mu_5+\nu_4+\nu_6,
    \mu_4+\mu_6+\nu_2+\nu_5, 
    \mu_3+\mu_5+\nu_4+\nu_5,
    \mu_4+\mu_5+\nu_3+\nu_5, 
    \mu_3+\mu_5+\nu_3+\nu_6,
    \mu_3+\mu_6+\nu_3+\nu_5, 
    \mu_2+\mu_6+\nu_3+\nu_6,
    \mu_3+\mu_6+\nu_2+\nu_6, 
    \mu_3+\mu_4+\nu_4+\nu_6,
    \mu_4+\mu_6+\nu_3+\nu_4, 
    \mu_5+\mu_6+\nu_2+\nu_4)$

    $\lambda_2+\lambda_3 \ge 
    \max(\mu_2+\mu_3+\nu_5+\nu_6,
    \mu_2+\mu_4+\nu_4+\nu_6,
    \mu_4+\mu_6+\nu_2+\nu_4,
    \mu_3+\mu_4+\nu_4+\nu_5,
    \mu_4+\mu_5+\nu_3+\nu_4,
    \mu_2+\mu_5+\nu_3+\nu_6,
    \mu_3+\mu_6+\nu_2+\nu_5,
    \mu_3+\mu_5+\nu_3+\nu_5,
    \mu_2+\mu_6+\nu_2+\nu_6, 
    \mu_5+\mu_6+\nu_2+\nu_3)$    
    
    $\lambda_1+\lambda_6 \ge 
    \max(\mu_1+\mu_6+\nu_5+\nu_6,
    \mu_2+\mu_6+\nu_4+\nu_6,
    \mu_4+\mu_6+\nu_2+\nu_6,
    \mu_3+\mu_6+\nu_3+\nu_6, 
    \mu_5+\mu_6+\nu_1+\nu_6)$
    
    $\lambda_1+\lambda_5 \ge 
    \max(\mu_1+\mu_5+\nu_5+\nu_6,
    \mu_2+\mu_5+\nu_4+\nu_6,
    \mu_4+\mu_6+\nu_2+\nu_5,
    \mu_1+\mu_6+\nu_4+\nu_6,
    \mu_4+\mu_6+\nu_1+\nu_6,    
    \mu_2+\mu_6+\nu_4+\nu_5,
    \mu_4+\mu_5+\nu_2+\nu_6,
    \mu_3+\mu_6+\nu_3+\nu_5,
    \mu_3+\mu_5+\nu_3+\nu_6,
    \mu_3+\mu_6+\nu_2+\nu_6,
    \mu_2+\mu_6+\nu_3+\nu_6,
    \mu_5+\mu_6+\nu_1+\nu_5)$

    $\lambda_1+\lambda_4 \ge 
    \max(\mu_1+\mu_4+\nu_5+\nu_6,
    \mu_2+\mu_4+\nu_4+\nu_6,
    \mu_4+\mu_6+\nu_2+\nu_4,
    \mu_1+\mu_5+\nu_4+\nu_6,
    \mu_4+\mu_6+\nu_1+\nu_5,
    \mu_2+\mu_5+\nu_4+\nu_5,
    \mu_4+\mu_5+\nu_2+\nu_5,
    \mu_3+\mu_6+\nu_3+\nu_4,
    \mu_3+\mu_4+\nu_3+\nu_6,
    \mu_3+\mu_6+\nu_2+\nu_5,
    \mu_2+\mu_5+\nu_3+\nu_6,
    \mu_3+\mu_6+\nu_1+\nu_6,
    \mu_1+\mu_6+\nu_3+\nu_6,
        \mu_3+\mu_5+\nu_2+\nu_6,
    \mu_2+\mu_6+\nu_3+\nu_5,    
    \mu_2+\mu_6+\nu_2+\nu_6,
    \mu_3+\mu_5+\nu_3+\nu_5,
    \mu_5+\mu_6+\nu_1+\nu_4)$

    $\lambda_1+\lambda_3 \ge 
    \max(\mu_1+\mu_3+\nu_5+\nu_6,
    \mu_2+\mu_3+\nu_4+\nu_6,
    \mu_4+\mu_6+\nu_2+\nu_3,
    \mu_1+\mu_4+\nu_4+\nu_6,
    \mu_4+\mu_6+\nu_1+\nu_4,
    \mu_2+\mu_4+\nu_4+\nu_5,
    \mu_4+\mu_5+\nu_2+\nu_4,
    \mu_3+\mu_6+\nu_2+\nu_4,
    \mu_2+\mu_4+\nu_3+\nu_6,
    \mu_3+\mu_6+\nu_1+\nu_5,
    \mu_1+\mu_5+\nu_3+\nu_6,
    \mu_3+\mu_5+\nu_3+\nu_4,
    \mu_3+\mu_4+\nu_3+\nu_5,
    \mu_3+\mu_5+\nu_2+\nu_5,
    \mu_2+\mu_5+\nu_3+\nu_5,    
        \mu_2+\mu_6+\nu_1+\nu_6,
    \mu_1+\mu_6+\nu_2+\nu_6,
    \mu_2+\mu_6+\nu_2+\nu_5,
    \mu_2+\mu_5+\nu_2+\nu_6,
    \mu_5+\mu_6+\nu_1+\nu_3)$

    $\lambda_1+\lambda_2 \ge 
    \max(\mu_1+\mu_2+\nu_5+\nu_6,
    \mu_1+\mu_3+\nu_4+\nu_6,
    \mu_4+\mu_6+\nu_1+\nu_3,
    \mu_2+\mu_3+\nu_4+\nu_5,
    \mu_4+\mu_5+\nu_2+\nu_3,
    \mu_3+\mu_6+\nu_1+\nu_4,
    \mu_1+\mu_4+\nu_3+\nu_6,
    \mu_3+\mu_5+\nu_2+\nu_4,
    \mu_2+\mu_4+\nu_3+\nu_5,
    \mu_3+\mu_4+\nu_3+\nu_4,    
    \mu_2+\mu_6+\nu_1+\nu_5,    
    \mu_1+\mu_5+\nu_2+\nu_6,
    \mu_2+\mu_5+\nu_2+\nu_5,
    \mu_1+\mu_6+\nu_1+\nu_6,
    \mu_5+\mu_6+\nu_1+\nu_2)$
    \item [$s=5$:] $\lambda_1 \ge \max(\mu_6+\nu_1,\mu_1+\nu_6, \mu_2+\nu_5,\mu_5+\nu_2, \mu_3+\nu_4,\mu_4+\nu_3)$
    
    $\lambda_2 \ge \max(\mu_6+\nu_2,\mu_2+\nu_6, \mu_3+\nu_5,\mu_5+\nu_3, \mu_4+\nu_4)$
    
    $\lambda_3 \ge \max(\mu_6+\nu_3,\mu_3+\nu_6, \mu_4+\nu_5,\mu_5+\nu_4)$

    $\lambda_4 \ge \max(\mu_6+\nu_4,\mu_4+\nu_6, \mu_5+\nu_5)$

    $\lambda_5 \ge \max(\mu_6+\nu_5,\mu_5+\nu_6)$

    $\lambda_6 \ge \max(\mu_6+\nu_6)$

\end{enumerate}

\end{document}